\documentclass[12pt,english,letterpaper]{article}
\usepackage[english]{babel}
\usepackage{amsmath,amsthm,amssymb}

\usepackage{latexsym}
\usepackage{enumerate}
\usepackage[utf8]{inputenc}
\usepackage{color}

\usepackage{xcolor}
\usepackage{graphicx}
\usepackage{amsmath}
\usepackage{amsfonts}
\usepackage{amsthm,amscd}
\usepackage{amsmath}
\usepackage{amssymb}
\usepackage{amstext}
\allowdisplaybreaks

\newtheorem{theorem}{Theorem}
\newtheorem*{theorem*}{Theorem}

\newtheorem{definition}[theorem]{Definition}
\newtheorem{example}[theorem]{Example}
\newtheorem{lemma}[theorem]{Lemma}

\newtheorem{remark}[theorem]{Remark}

\newcommand{\Ortoex}

\newcommand{\norm}[1]{\left\lVert #1 \right\rVert}
\newcommand{\abs}[1]{\left| #1 \right|}

\begin{document}

\title{A noncommutative Ruelle's Theorem for a normalized potential taking values on positivity-improving operators}

\author{William  M. M. Braucks and  Artur O. Lopes}

\maketitle
\begin{abstract}

Let \(\mathcal{A}\) be a finite-dimensional real (or complex) C*-algebra, \(\Omega_{A}\) an aperiodic subshift of finite type, and \(\mathcal{C}(\Omega_{A}; \mathcal{A})\) the set  of continuous functions from
\(\Omega_{A}\) to \(\mathcal{A}\).  The shift $\sigma$ provides dynamics. Given a real Lipschitz potential $\varphi\in \mathcal{C}(\Omega_{A}; \mathfrak{L}(\mathcal{A}))$, where $\mathfrak{L}(\mathcal{A})$ is the set of linear operators acting on a real $\mathcal{A}$, we introduce a noncommutative analogue of Ruelle's 
operator, which acts on \(\mathcal{C}(\Omega_{A}; \mathcal{A})\).  Assuming the positivity-improving hypothesis, we prove  a version of Ruelle's Theorem   whenever the operator   is normalized. An  eigenstate (a linear functional) invariant for the  action of the noncommutative Ruelle's 
operator  will play the role of the Gibbs probability of  Thermodynamic Formalism; to be called a  Gibbs eigenstate.  We introduce the concept of entropy for a Gibbs eigenstate (obtained from a certain family of potentials $\varphi$) -   generalizing  the classical one. In our setting, there is currently no direct relationship with cocycles and Lyapunov exponents.  We present  examples illustrating the novelty of the cases that can be considered, ranging  from topics related to quantum channels to Pauli matrices. Interesting cases:  $\mathcal{A}=M_{N \times N}(\mathbb{R})$ and  $\mathcal{A}= \mathbb{R}^{N}$.

\end{abstract}

AMS 2010 Classification Number - 37D35; 94A40

\section{Introduction}

Ruelle's Theorem plays an important role in the celebrated thermodynamic formalism approach to classical statistical
mechanics \cite{ruelle-rig,craizer-m,lopes-nft,mane-ntf,papo,walters}. It provides us with a very useful global
description of the equilibrium states. These states are traditionally represented by probability measures on the
configuration space. We denote by $\sigma$ the shift and by $\Omega_A\subset\{1,2,...,k\}^\mathbb{N}$ a subshift of finite type (see Definition \ref{fifi}).
Such stationary probabilities for the shift describe the long-term statistical behavior of the system under a given
interaction potential $B: \Omega_A \to \mathbb{R}$, that is, the likelihood of observing particular configurations in the thermodynamic regime. The potential $B$ here corresponds  to the negative of the Hamiltonian of Statistical Mechanics.

 Since we are primarily interested in obtaining results involving real quantities, such as entropy, etc., we chose to describe the model using a real $C^*$-algebra,
	but the formulation using a complex $C^*$-algebra would also be compatible with our reasoning.

Unlike the classical approach, commonly used in statistical mechanics, which analyzes finite subsystems and then takes a thermodynamic limit, Ruelle's
Theorem identifies equilibrium states directly; Kolmogorov-Sinai entropy  for the shift plays an important role. It characterizes them as the product of an eigenfunction and an
eigenmeasure of a positive linear operator on continuous functions, known as the transfer (or Ruelle) operator. This
Theorem also helps to  establish an associated  variational principle, according to which the topological pressure is given by the
Legendre–Fenchel transform of the Kolmogorov-Sinai entropy (which is an affine function). It is well known that the study of  equilibrium states for Hölder potentials defined on $\{1,2,...,k\}^\mathbb{N}$ helps to analyze  equilibrium states for Hölder potentials defined on $\{1,2,...,k\}^\mathbb{Z}$ (see Section 1 in \cite{papo}).

Ruelle's Theorem for a potential $B:\Omega_A \to \mathbb{R}$ is a dynamical counterpart of the Perron-Frobenius Theorem, which, as noted in \cite{evkr}, may be
regarded as concerning the spectral properties of positive
linear operators on finite-dimen\-sional commutative C*-algebras.
%

To clarify this interpretation, recall, from Gelfand duality \cite[Theorem 2.1.11]{brro}, that the
finite-dimensional commutative C*-algebras are precisely those of the form \(\mathbb{R}^{n}\). For a linear
self-mapping of \(\mathbb{R}^{n}\), being positive is equivalent to its matrix having positive entries. Note that we are
referring here to the notion of positivity in the sense of preserving a positive cone, not to the notion of
(semi-)definite positivity of a linear operator on a Hilbert space. This is particularly important regarding the linear
self-mappings of \(\mathbb{R}^{n}\), since this set can be seen both as a finite-dimensional commutative C*-algebra and
as a Hilbert space. Note also that both the matrix which represents a linear self-mapping of \(\mathbb{R}^{n}\) and
the positive cone \(\mathbb{R}^{+n}\) depend on the choice of a basis.

Extending this interpretation, we may regard Ruelle's Theorem as replacing finite-dimensional commutative real C*-algebras
by infinite-dimensional real commutative C*-algebras of the form \(\mathcal{C}(\Omega_{A}; \mathbb{R})\), and positive
matrices by transfer operators. This observation is the starting point for the version of Ruelle's Theorem we are going
to prove here.

It is also relevant to note that there are many other successful versions of the Perron-Frobenius theory that have been proved for
different  settings; see, for instance
\cite{nc-rto,kreinrutman,patger-irr,ts-core,nuss-kr,ch-nonlinear,zhang-kr,faren-irr,brknlo-epgcgp,Arb,Miha}.

It is natural to look for noncommutative analogues of Ruelle's Theorem, which can be eventually  used in  applications in quantum statistical
mechanics  (see, for instance  \cite{evkr}, \cite{bara}, \cite{Vwe}, \cite{brknlo-epgcgp}, \cite{brknlo-epgcgp1}, \cite{Mayer} or \cite{Lar}). Perhaps the most successful example was Matsui's quantum channel formalism, inaugurated in \cite{nc-rto}.
There, it is shown that an analogue of Ruelle's Theorem holds for a wide class of transfer operators on UHF\footnote{Uniformly
HyperFinite (see \cite{ta-theory-i,da-example,rolala}). For example, the CAR algebra.} algebras, which are a family
of noncommutative C*-algebras associated with infinite chains of quantum spin particles; our setting  is somewhat different.

In the present paper, we prove a version of Ruelle's Theorem for a class of finite-dimensional  eventually noncommutative
C*-algebras different from those of \cite{nc-rto}.
More precisely: for the class of C*-algebras of continuous functions of aperiodic subshifts of finite
type\footnote{Sometimes called topological Markov shifts.} into finite-dimensional real $C^*$-algebras. In order to do so, we borrow
several definitions and techniques from \cite{papo,walters,craizer-m,mane-ntf,lopes-nft,olivi-fte}.

In the classical Ruelle theorem for a positive Lipschitz  potential $B:\Omega_A \to (0,\infty)$, the interaction potential enters the transfer operator through a strictly
positive scalar function in exponentiated form. This strict positivity of $B$ plays a fundamental role: it ensures
that the transfer operator eventually maps non-zero positive functions to strictly positive ones, thereby enabling
uniqueness and convergence results for eigenfunctions and eigenmeasures.
This is an important issue to consider when attempting to obtain noncommutative generalizations of classical theory.

In our noncommutative setting, we sought a natural generalization of this structure. While it is tempting to
consider positive or even completely positive operator-valued potentials — the latter being a widely accepted
generalization of scalar positivity \cite{choi-schw,choi-comp,st-pfun,st-pos,arv-sub,arv-sub2} —
these notions alone are insufficient for our purposes. In particular, they do not guarantee that the transfer
operator satisfies the conditions needed in \cite[Theorem 4.1, hypothesis (a)]{glwe}, which require a form of
strict positivity.

We therefore impose the alternative condition that the potential takes values in positivity-improving linear mappings:
that is, mappings which send every non-zero positive element to a strictly positive one. Notice that this condition is
stronger than positivity, but does not imply nor is implied by complete positivity. However, it captures the same
essential structure present in the classical case, and enables us to reproduce the key features of the Ruelle
Theorem in this more general, noncommutative context. Quantum channels  provide  examples of linear transformations acting on the set of matrices and preserving the cone of positive matrices (see \cite{nich}). Explicit examples of positivity-improving linear mappings acting on the space of $n$ by $n$ matrices  appear in Section 8 in \cite{brknlo-epgcgp}. We will consider here functions defined on  the symbolic space $\Omega_A$ taking values  in positivity-improving linear mappings.

%
%

To fully understand the significance of the results obtained here, it is necessary to describe first some aspects of the classical Ruelle Theorem as described in \cite{papo}. Consider a Hölder (or Lipschitz) positive  potential $J: \Omega_A \to (0,\infty) \subset\mathbb{R}$, and the associated Ruelle operator $\mathcal{L}_J$, which acts on continuous functions $f: \Omega_A \to \mathbb{R}$ via $\mathcal{L}_J(f)(x) =\sum_{\sigma(y)=x} J(y)f(y).$ The operator $\mathcal{L}_J$ preserves the set of continuous positive functions.  We say that $J$ is normalized if $\mathcal{L}_J(1)=1$; in this case, let $\mathcal{L}_J^*$   be the dual operator, which acts on the set of probabilities. A probability $\mu_J$ that satisfies the fixed-point property for $\mathcal{L}_J^*$  is called a Gibbs probability (sometimes called Gibbs state) for the normalized potential $B$. The probability $\mu_J$ is an eigenprobability for the operator $\mathcal{L}_J^*$ associated with the eigenvalue $1$.
The set of  probabilities in this class is invariant for the shift action and is one of the main focuses of research in thermodynamic formalism. They describe the possible equilibrium states for Hölder (or Lipschitz) potentials defined on $\{1,2,...,k\}^\mathbb{N}$. The maximal entropy measure $\mu$ is the fixed point for the action of the operator $\mathcal{L}_{\frac{1}{k}}^*$. 

We will denote by $\mathfrak{L}(\mathcal{A})$ the set of linear operator acting on the finite dimensional C*-algebra $\mathcal{A}$.
Given a Lipschitz potential $\varphi: \Omega_A \to \mathfrak{L}(\mathcal{A})$,
we will introduce a noncommutative analogue of the Ruelle operator, which will be denoted by $L_\varphi$; we will demonstrate, in a noncommutative setting, the existence of the analog of  Hölder Gibbs states, and these will be denoted generically by linear operators $\eta$, which will be called  eigenstates. In some examples (see sections \ref{exao} and \ref{some}), we will explicitly exhibit such an eigenstate operator $\eta$; in some of them, a {\it classical} Lipschitz equilibrium probability (here called  Gibbs probability) will help to express the precise form of this eigenstate $\eta$. We will assume that the noncommutative Ruelle operator is normalized (see Definition \ref{nonai}). An important issue, exemplified  in Example \ref{ex:max-ent}, is the fact  that scalar probability measures are
insufficient to describe equilibrium in the noncommutative setting. The variety of examples one can get ranges from topics related to quantum channels to Pauli matrices (see \cite{nich}, \cite{Benoi}, \cite{brknlo-epgcgp} or \cite{brknlo-epgcgp1} for definitions). In Example \ref{cf}  the potential $\varphi: \{1,2\}^\mathbb{N} \to \mathcal{L} (  M_{2 \times 2}(\mathbb{R} ))$ is given by a fixed quantum channel acting on the family of continuous functions $g:\{1,2\}^\mathbb{N} \to M_{2 \times 2}(\mathbb{R})$. In a particular case one can take $\mathcal{A}= \mathbb{R}^{N}$ and $\varphi$  of the form \(\varphi : \Omega_{A} \to M_{N \times N}(\mathbb{R})\) (see Example \ref{er2}).

Given the normalized  noncommutative Ruelle operator $L_\varphi$,  the associated
eigenstate \(\eta\) will be the one such that is invariant for the action of this Ruelle operator (as for instance in \eqref{yst}). One can show that such $\eta$ also satisfy a form of shift-invariance (see  \eqref{efen} in Section \ref{Pre}). It would be natural to call such a state of Gibbs eigenstate.

We believe that at this point an explicit example will help to understand the generality of what is being proposed here (we will elaborate on more details  in Example \ref{ex:max-ent}). This is perhaps the simplest of all possible examples. Take $\mathcal{A}= M_{2 \times 2}(\mathbb{R})$ and denote a general  matrix-valued function by \(g \in \mathcal{C}(\Omega_{A}; M_{2 \times 2}(\mathbb{R}))\). Given a normalized  Lipschitz potential $\varphi$ taking values in linear operators \(M_{2 \times 2}(\mathbb{R}) \to M_{2 \times 2}(\mathbb{R})\), we can introduce a normalized noncommutative Ruelle operator $L_\varphi$, which will act on functions $g$ of the above form.
A state \(\eta\) acting of functions $g$ is  a linear functional:
$$\eta : \mathcal{C}(\left\{1,2\right\}^{\mathbb{N}}; M_{2 \times 2}(\mathbb{R})) \to \mathbb{R}.$$ For the case of the noncommutative Ruelle operator $L_\varphi$ described in Example \ref{ex:max-ent}, the associated eigenstate $\eta$ is  defined by:
\begin{equation*}
g \mapsto \eta(g) := \int \widehat{\operatorname{tr}}\, (g(x)) \operatorname{d} \! \mu(x) \text,
\end{equation*}
where \(\mu\) denotes the maximal entropy measure for the action of the full shift on on $\{1,2\}^\mathbb{N}$, and $\widehat{\operatorname{tr}} := \frac{1}{2} \operatorname{tr}$ is the normalized trace. We may call such $\eta$ of Gibbs  eigenstate.

In the classical setting, one way to obtain an approximation of the Gibbs state $\mu_J$ for the potential $J$, would be to take the limit in $n$ of the iterations of the dual  operator $\mathcal{L}_J^*$. More precisely, given any probability $\rho$, then $\mu_J= \lim_{n\to \infty} (\mathcal{L}_J^*)^n(\rho)$. In fact, one can show that the action of $\mathcal{L}_J^*$ is a contraction for the $1$-Wasserstein distance  (see \cite{KLS}).  Our main theorem will ensure that one can obtain the eigenstate $\eta$ as a limit in a similar procedure; we will work out some examples.

In Sections \ref{exao} and  \ref{some}, we will present examples that illustrate the novelty of our results. In Section \ref{mma}, we will present the proof of our main theorem, which is a noncommutative version of the Ruelle Theorem. In Section \ref{Pre} we will introduce a form of entropy for  eigenstates associated with a certain family of normalized potentials, the so-called {\it trace-type potentials} (see Definitions \ref{efen},  \ref{def:tr-type} and Theorem \ref{prep}).
In Section \ref{rela}, at the end of the paper, we will comment on similarities and dissimilarities with some previous works. We claim that, to the best of our knowledge, our results are not included in the statements  of other works in the literature.

The analysis of the case when $\mathcal{A}$ has infinite dimension introduces several technical difficulties and will be left for future work.

We note that our results are of a different nature from those of the linear cocycle and Lyapunov exponents theory, and that they are described in a vast literature; it is worth listing  here some important examples: \cite{AvVi}, \cite{Bochi}, \cite{back1}, \cite{back2}, \cite{Feng}, \cite{Varao} and \cite{Morris1}. f  We do not rule out the possibility that in some future work, Lyapunov exponents and Oseledet's theorem may be relevant for obtaining results  to the present setting.

The concept of entropy described here, associated with the use of a normalized generalized Ruelle operator, is quite distinct from those considered in other texts regarding C*-algebras, such as the following: \cite{En1}, \cite{En4}, \cite{En3}, \cite{En2}, \cite{Mayer} and \cite{En5}.

\section{Preliminaries and definitions}

Let \(\Sigma = \left\{1, 2, \cdots, k\right\}\) and \(A : \Sigma \times \Sigma \to \left\{0, 1\right\}\) be
a matrix of zeroes and ones. We call elements \(i, j \in \Sigma\) symbols and \(A\) the transition matrix. When
\(A(i, j) = 1\) we say that \((i, j)\) is allowed and when \(A(i, j) = 0\) we say that \((i, j)\) is forbidden. In
what follows, \(A\) is always taken to be aperiodic.
\begin{definition} \label{fifi}
The subshift of finite type \(\Omega_{A}\) is given by:
\begin{equation*}
\Omega_{A} := \left\{ (x_{n})_{n \in \mathbb{N}}\, \middle| \,x_{n} \in \Sigma, \, A(x_{n}, x_{n+1}) = 1, \, \forall n \in
\mathbb{N}\right\} \text.
\end{equation*}
It is the set of all unilaterally infinite sequences on \(k\) symbols with transitions allowed by \(A\).
\end{definition}

\begin{definition}
We consider \(\Omega_{A}\) to be equipped with the product topology\footnote{Sometimes called Tychonoff product topology.}
and denote the set of continuous functions from \(\Omega_{A}\) to a topological space \(X\) by \(\mathcal{C}(\Omega_{A}; X)\).
\end{definition}

\begin{definition}
Let \(X\) be a Banach space. We denote its norm by \(\abs{ \cdot }_{X}\). We also consider \(\mathcal{C}(\Omega_{A}; X)\)
to be equipped with the uniform norm, and denote it by:
\begin{equation*}
\abs{g}_{\infty} := \sup_{x \in \Omega_{A}} \abs{g}_{X} \text,
\end{equation*}
which makes it a Banach space. Therefore, we mildly abuse notation, since \(\abs{ \cdot }_{\mathcal{C}(\Omega_{A}; X)} := \abs{ \cdot }_{\infty}\).
In each case,  \(X\) will always be clear from context whenever we employ the notation \(\abs{ \cdot }_{\infty}\).
\end{definition}

\begin{definition}
Let \(X\) be a Banach space. We denote its topological dual space by \(X^{\prime}\). We also consider \(X^{\prime}\)
to be equipped with the weak* norm, which makes it a Banach space.
\end{definition}

\begin{definition}
Let \(X, Y\) be two Banach spaces, \(T : X \to Y\) a bounded linear mapping. We denote the transpose of \(T\) by
\(T^{\prime} : Y^{\prime} \to X^{\prime}\). It is also a bounded linear mapping.
\end{definition}

\begin{definition}
Let \(0 < \theta < 1\) be an arbitrary real parameter. The metric \(d_{\theta} : \Omega_{A} \times \Omega_{A} \to
\mathbb{R}^{+}\) is given by:
\begin{align*}
d_{\theta}(x, y) & := \theta^{-N(x,y)} \text{, where:} \\
N(x,y) & := \min \left\{n \in \mathbb{N} \middle| x_{n} \neq y_{n}\right\} \text.
\end{align*}
\end{definition}

\begin{remark}
The metric \(d_{\theta}\) induces the product topology for any \(0 < \theta < 1\), and $\Omega_A$ is a compact set..
\end{remark}

\begin{definition}
Whenever \(X\) is a Banach space, we denote the set of Lipschitz functions from \(\Omega_{A}\) to \(X\) by
\(\mathcal{C}_{\theta}(\Omega_{A}; X) \subseteq \mathcal{C}(\Omega_{A}; X)\). We consider it to be equipped with the
Lipschitz seminorm:
\begin{equation*}
\abs{g}_{\theta} := \sup_{x \neq y \in \Omega_{A}} \frac{\abs{g(x) - g(y)}_{X}}{d_{\theta}(x,y)} \text,
\end{equation*}
and the Lipschitz norm:
\begin{equation*}
\norm{g}_{\theta} := \abs{g}_{\infty} + \abs{g}_{\theta} \text.
\end{equation*}
\end{definition}

\begin{remark} \label{re:arz-asc}
When $X$ is finite-dimensional, the Arzelà-Ascoli Theorem ensures that every bounded subset of
\(\left(\mathcal{C}_{\theta}(\Omega_{A}; X), \norm{\cdot}_{\theta}\right)\) is relatively compact in
\(\left(\mathcal{C}(\Omega_{A}; X), \abs{\cdot}_{\infty}\right)\).
\end{remark}

\begin{definition}
The shift mapping \(\sigma : \Omega_{A} \to \Omega_{A}\) is given by:
\begin{equation*}
\sigma(x) = \sigma((x_{n})_{n \in \mathbb{N}}) := (x_{n+1})_{n \in \mathbb{N}} \text.
\end{equation*}
\end{definition}

\begin{remark}
The shift mapping is expanding with respect to the metric \(d_{\theta}\) for any \(0 < \theta < 1\). The expanding parameters
may be taken \(\theta^{-1} > 1\) and \(\theta > 0\). Since \(A\) is aperiodic, the shift mapping is also topologically mixing.
\end{remark}

\begin{remark}
Let \(\mathcal{A}\) be a C*-algebra. Then \(\mathcal{C}(\Omega_{A}; \mathcal{A})\) is also a C*-algebra. It is commutative if, and only if,
\(\mathcal{A}\) is. It is unital if, and only if, \(\mathcal{A}\) is. Notice it is infinite-dimensional except possibly if \(k = 1\),
when it is isomorphic to \(\mathcal{A}\), which may be finite-dimensional. The reader should be mindful of this through the
following definitions.
\end{remark}

\begin{definition}
Let \(\mathcal{A}\) be a real C*-algebra. We say that an element \(a \in \mathcal{A}\) is positive when either of the following equivalent
conditions hold:
\begin{enumerate}
\item \(a = b^{\ast}b\) for some \(b \in \mathcal{A}\).

\item \(a = a^{\ast}\) and \(\operatorname{spec}(a) \subseteq [0, +\infty)\).
%
\end{enumerate}
In such a case, \(b \in \mathcal{A}\) can be chosen \(b = \abs{a} = \sqrt{a^{\ast}a}\), which is also positive, taking into
account the functional calculus \cite[Section 4.6.1]{bp-qsm}. Then, we write \(a \geq 0\).
\end{definition}

\begin{definition}
The set of all \(\mathcal{A} \ni a \geq 0\) is called the positive cone of \(\mathcal{A}\), and written \(\mathcal{A}^{+}\).
\end{definition}

\begin{remark}
The positive cone of \(\mathcal{C}(\Omega_{A}; \mathcal{A})\), written \(\mathcal{C}(\Omega_{A}; \mathcal{A})^{+}\), is equal to the set of
continuous functions from \(\Omega_{A}\) to \(\mathcal{A}^{+}\), written \(\mathcal{C}(\Omega_{A}; \mathcal{A}^{+})\).
\end{remark}

\begin{definition}
Let \(\mathcal{A}\) be a unital C*-algebra. We write \(\mathbb{I}_{\mathcal{A}}\) for the multiplicative unit of \(\mathcal{A}\). To simplify
the notation, we also write \(\mathbb{I}_{\mathcal{C}(\Omega_{A}; \mathcal{A})} := \mathbb{I}_{\mathcal{C}}\). Notice this is merely the constant
function \(\mathbb{I}_{\mathcal{C}} : \Omega_{A} \to \mathcal{A}\), that maps every point \(x \in \Omega_{A}\) to the multiplicative unit
of \(\mathcal{A}\), that is: \(\mathbb{I}_{\mathcal{C}} \equiv \mathbb{I}_{\mathcal{A}}\), or \(x \overset{\mathbb{I}_{\mathcal{C}}}{\longmapsto} \mathbb{I}_{\mathcal{A}}\).
\end{definition}

\begin{definition}
Let \(\mathcal{A}\) be a unital C*-algebra. We say that an element \(a \in \mathcal{A}\) is strictly positive when there exists
\(\mathbb{R} \ni \varepsilon > 0\) such that \(a \geq \varepsilon \mathbb{I}_{\mathcal{A}}\). Then, we write \(a > 0\).
\end{definition}

\begin{remark}
The set of all \(\mathcal{A} \ni a > 0\) is the interior of the positive cone of \(\mathcal{A}\), and therefore written \(\mathcal{A}^{+^{\circ}}\).
\end{remark}

\begin{remark}
Unless \(\mathcal{A} = \mathbb{R}\), the conjunction \(a \in \mathcal{A}^{+}\) plus \(a \neq 0\) does not imply \(a \in \mathcal{A}^{+^{\circ}}\).
\end{remark}

\begin{remark}
However, the conjunction \(a \in \mathcal{A}^{+}\) plus there exists \(a^{-1}\) such that \(a^{-1}a = \mathbb{I}_{\mathcal{A}}\) does imply \(a \in \mathcal{A}^{+^{\circ}}\).
Therefore, the strictly positive elements of a unital C*-algebra \(\mathcal{A}\) are exactly the positive units\footnote{
Units are elements possessing a multiplicative inverse.} of \(\mathcal{A}\).
\end{remark}

\begin{definition}
Let \(\mathcal{A}\) be a real $C^*$-algebra. We denote the set of linear mappings from \(\mathcal{A}\) to itself by \(\mathfrak{L}(\mathcal{A})\). We consider
it to be equipped with the operator norm.
\end{definition}

\begin{definition} \label{def:str-pos}
Let \(\mathcal{A}\) be a real C*-algebra. We say that a linear mapping \(\varphi \in \mathfrak{L}(\mathcal{A})\) is positivity-improving when \(\varphi(a) > 0\) for any \(\mathcal{A}^{+} \ni a \neq 0\).
\end{definition}

\begin{remark}
The composite of positivity-improving linear mappings is a positivity-improving linear mapping.
\end{remark}

\begin{remark}
The sum of positivity-improving linear mappings is a positivi\-ty-improving linear mapping.
\end{remark}

\begin{definition}
Let \(\Omega_{A}\) be a subshift of finite type and \(\mathcal{A}\) a finite-dimension\-al real C*-algebra. We say that a function
\(\varphi : \Omega_{A} \to \mathfrak{L}(\mathcal{A})\), written \(x \mapsto \varphi_{x}\), is a potential when
it takes values in positivity-improving linear mappings and is Lipschitz-continuous with respect to the operator
norm.
\end{definition}

\begin{definition}
Let \(\Omega_{A}\) be a subshift of finite type, \(\mathcal{A}\) a real  finite-dimensional C*-algebra, and \(\varphi : \Omega_{A} \to \mathfrak{L}(\mathcal{A})\)
a potential. The transfer operator \(L_{\varphi}\) acts on continuous functions \(g \in \mathcal{C}(\Omega_{A}; \mathcal{A})\) by:
\begin{equation} \label{eq:def-transf}
\left( L_{\varphi} g \right)(x) := \sum_{y \in f^{-1}(x)} \varphi_{y} \big( g(y) \big) \text.
\end{equation}
\end{definition}

\begin{definition} \label{nonai}
Let \(\varphi : \Omega_{A} \to \mathfrak{L}(\mathcal{A})\) be a potential. We say that \(\varphi\) is normalized if it is such that:
\begin{equation*}
L_{\varphi}(\mathbb{I}_{\mathcal{C}}) = \mathbb{I}_{\mathcal{C}} \text,
\end{equation*}
or, equivalently:
\begin{equation} \label{eq:p-w-normal}
\sum_{\substack{i \in \Sigma \\ A(i,x_{1}) = 1}} \varphi_{ix} \mathbb{I}_{\mathcal{A}} = \mathbb{I}_{\mathcal{A}} \text, \quad \forall \, x \in \Omega_{A} \text.
\end{equation}
\end{definition}

\begin{remark}
Equation \eqref{eq:p-w-normal} plus the positivity of every \(\varphi_{x}\) allows us to conclude that \(0 \leq \varphi_{x}(\mathbb{I}_{\mathcal{A}}) \leq \mathbb{I}_{\mathcal{A}}\).
Also, \cite[Corollary 2.9 (Russo-Dye)]{paul} applies. Therefore, \(\norm{\varphi_{x}} \leq 1\) for any \(x \in \Omega_{A}\).
\end{remark}
\medskip

We will  provide a simple example to illustrate the setting we consider here.

\begin{example}
Let $\mathcal{A}= M_{2 \times 2}(\mathbb{R})$, $\Omega=\{1,2\}^\mathbb{N}$, and $\sigma$ the full shift.
Consider a Lipschitz function  $\varphi : \Omega=\{1,2\}^\mathbb{N} \to \left[ M_{2 \times 2}(\mathbb{R}) \to M_{2 \times 2}(\mathbb{R}) \right]=\mathfrak{L}(\mathcal{A})$ ,
such that there exist strictly positive matrices $A_1,A_{2}: \Omega \to  M_{2 \times 2}(\mathbb{R})$, satisfying for any $y\in \Omega$
\begin{equation} \varphi_{(1,y)} = \left( \widehat{\operatorname{tr}}\,( \cdot ) \right) A_{1} \,\, \text{and}\,\, \varphi_{(2,y)} = \left( \widehat{\operatorname{tr}}\,( \cdot ) \right) A_{2} ,
\end{equation}
$$A_1 + A_{2} =I.$$

The corresponding transfer operator \(L_{\varphi}\) acts on a general continuous function $g: \Omega \to  M_{2 \times 2}(\mathbb{R}) $, in such way that $L_\varphi(g)=f$, via

$$ g \,\mapsto [\,y \to f(y)=L_\varphi(g)(y)\,] =[\,y \to\frac{1}{2} (\,\widehat{\operatorname{tr}}\,  g(1,y)\,) \, A_1 + \frac{1}{2} (\,\widehat{\operatorname{tr}}\,  g(2,y)\,) \, A_{2}\,] \text,$$
which means in a short notation

$$ g \,\mapsto L_\varphi(g)(y) =\frac{1}{2} (\,\widehat{\operatorname{tr}}\,  g(1,y)\,) \, A_1 + \frac{1}{2} (\,\widehat{\operatorname{tr}}\,  g(2,y)\,) \, A_{2} \text.$$

Note that the potential \(\varphi\) is normalized and positivity-improving. In this case,
we will be interested in the state  \(\eta : \mathcal{R}(\Omega_{A}; M_{2 \times 2}(\mathbb{R})) \to \mathbb{R}\) that plays the role analogous to the  eigenprobability for the dual of the Ruelle operator.
In   \eqref{gg2} and  \eqref{itre} in Example \ref{spi1}   we will elaborate on this issue.

\end{example}

\begin{example} \label{er2}

Take \(\mathcal{A} = \mathbb{R}^{N}\). Therefore, we have to consider potentials of the form \(\varphi : \Omega_{A} \to M_{N \times N}(\mathbb{R})= \mathfrak{L}( \mathcal{A})\), which will act on
functions \(g \in \mathcal{C}(\Omega_{A}; \mathbb{R}^{N})\).
We present several examples of such kind in Subsection \ref{ok3}.
The most simple one is to  take $\Omega=\{1,2\}^\mathbb{N}$, and $\sigma$ the full shift. The multiplicative element in \(\mathcal{A} = \mathbb{R}^{2}\) is the vector $(1,1)$, therefore the constant potential, $\forall x$
$$\varphi(x) = \begin{bmatrix}
1/4\, & \,1/4 \\
1/4 & 1/4
\end{bmatrix}  $$
is normalized. Given $x\in \Omega \to g(x)=(g_1(x),g_2(x))$, we get
$$ g \,\to L_\varphi(g)(y) =\begin{bmatrix}
1/4\, & \,1/4 \\
1/4 & 1/4
\end{bmatrix}(g_1(1x) + g_2(1x)+ g_1(2x) + g_2(2x))\,.$$

More details will be presented in Subsection \ref{ok3}.

\end{example}

Our main result is:
\smallskip

\begin{theorem} \label{thm:ruelle-normal}
Let \(\Omega_{A}\) be a subshift of finite type, \(\mathcal{A}\) a finite-dimensional C*-algebra, and \(\varphi : \Omega_{A} \to \mathfrak{L}(\mathcal{A})\)
a Lipschitz-continuous normalized potential. Then:
\begin{enumerate}
\item \(1 \in \mathbb{R}\) is a simple isolated eigenvalue of \(L_{\varphi}\).

\item All positive eigenfunctions of \(L_{\varphi}\) are positive multiples of \(\mathbb{I}_{\mathcal{C}}\).

\item
\begin{enumerate}
	\item There is a unique state \(\eta : \mathcal{C}(\Omega_{A}; \mathcal{A}) \to \mathbb{R}\) such that \(L_{\varphi}^{\prime}(\eta) = \eta\).
	It is faithful.
	
	\item \label{thmitm:unilim} For any \(g \in \mathcal{C}(\Omega_{A}; \mathcal{A})\), the sequence \(L_{\varphi}^{n}(g)\) converges uniformly
	to \(\eta(g)\mathbb{I}_{\mathcal{C}}\).
\end{enumerate}

\item The remainder of the spectrum of \(L_{\varphi}\), that is, \(\operatorname{spec}(L_{\varphi}) \setminus \left\{1\right\}\),
is contained in a disk  in the complex plane of radius strictly smaller than \(1\).
\end{enumerate}
\end{theorem}

It is natural to call the  $\eta$ mentioned  above in 3. (a) a Gibbs  eigenstate.

We present a few examples in the next section. We believe they will help to better understand the theoretical results being presented. More examples will be described at the end of the paper in Section \ref{some}.

\section{Some examples} \label{exao}

\begin{example} \label{ex:tr-type}
Let \(\Omega_{A} \subseteq \left\{1, 2, \cdots, k\right\}^{\mathbb{N}}\) be any aperiodic subshift of finite type and $\mathcal{A}=  M_{d \times d}(\mathbb{R})$.
We will define next    a Lipschitz potential $\varphi$. First, for each  $x\in \Omega_{A} $ consider \(\varphi_{x} : M_{d \times d}(\mathbb{R}) \to M_{d \times d}(\mathbb{R})\) such that:
\begin{equation} \label{gg4}
\varphi_{x}(g(x)) = \left( \widehat{\operatorname{tr}}\, g(x) \right) P(x),
\end{equation}
where $P(x)=(P_{i,j})_{i,j=1,...,d}(x)$ is a family of strictly positive matrices (so that each \(\varphi_{x}\) is positivity-improving).
Suppose that \(\varphi\) is normalized, that is:
\begin{equation}\label{cabe}
\sum_{\substack{i \in \Sigma \\ A(i,x_{1}) = 1}} P(ix) = \operatorname{Id} \text,
\end{equation}
for all \(x \in \Omega_{A}\).

Now, we  consider the potential \(\widehat{\operatorname{tr}}\,P\) in the classical sense of \cite{papo}. That is, \(\widehat{\operatorname{tr}}\,P(x)\)   is equal to $\frac{1}{d}\sum_{j=1}^d P_{jj}(x).$

Condition \eqref{cabe} implies that the \emph{scalar} potential \(\widehat{\operatorname{tr}}\,P\)
is classically normalized. Denoting by \(\mu_{\widehat{\operatorname{tr}}\,P}\) the classical equilibrium probability measure for the classical potential  \(\widehat{\operatorname{tr}}\,P\),
we claim that the eigenstate \(\eta : \mathcal{C}(\Omega_{A}; M_{d \times d}(\mathbb{R})) \to \mathbb{R}\), described in item 3. (a) in Theorem \ref{thm:ruelle-normal}, is given
by:
\begin{equation} \label{gg3}
\eta(g) := \int \widehat{\operatorname{tr}}\, g(x) \operatorname{d} \! \mu_{\widehat{\operatorname{tr}}\,P}(x).
\end{equation}

Indeed,
\begin{align*}
\eta(L_{\varphi} g) & = \int \widehat{\operatorname{tr}}\, \left[\left(L_{\varphi} g\right)(y)\right] \operatorname{d} \! \mu_{\widehat{\operatorname{tr}}\,P}(y) \\
& = \int \widehat{\operatorname{tr}}\, \left(\sum_{\substack{i \in \Sigma \\ A(i,y_{1}) = 1}} \left( \widehat{\operatorname{tr}}\, g(iy) \right) P(iy) \right) \operatorname{d} \! \mu_{\widehat{\operatorname{tr}}\,P}(y) \\
& = \int \sum_{\substack{i \in \Sigma \\ A(i,y_{1}) = 1}} \left( \widehat{\operatorname{tr}}\, g(iy) \right) \left(\widehat{\operatorname{tr}}\,P(iy)\right) \operatorname{d} \! \mu_{\widehat{\operatorname{tr}}\,P}(y) \\
& = \int \Big(L_{\widehat{\operatorname{tr}}\,P} \widehat{\operatorname{tr}}\, g\Big)(y) \operatorname{d} \! \mu_{\widehat{\operatorname{tr}}\,P}(y) 
 \text.
\end{align*}
\begin{equation} \label{yst} = \int \widehat{\operatorname{tr}}\, g(x) \operatorname{d} \! \mu_{\widehat{\operatorname{tr}}\,P}(x)= \eta(g) \,\,\,\,\,\,\,\,\,\,\,\,\,\,\,\,\,\,\,
\end{equation}

\end{example}

As  we will see, this example is related to Example \ref{ex:diagonals}.
\medskip

\begin{example} \label{cf}
	The present example  takes into account the reasoning of   \cite[Example 4]{bara},  \cite[Example 8.5]{brknlo-epgcgp} or \cite[Section 2.7]{lopes-immq}. We consider $\Omega=\{1,2\}^\mathbb{N}$ and  $\mathcal{A}=M_{2 \times 2}(\mathbb{R})$. 
	
	Before defining the potential $\varphi$, we will consider a certain operator $a \to\varepsilon(a \otimes\operatorname{Id}) $, which is a particular case of a quantum channel (see \cite{nich}).  Later, we will consider the potential $\varphi$ such that  $g \to\varepsilon(g \otimes\operatorname{Id})= \varphi (g) $.
In summary, the potential $\varphi$ is given by a fixed quantum channel acting on the family of continuous functions 
$$g:\{1,2\}^\mathbb{N} \to M_{2 \times 2}(\mathbb{R}).$$
	 The explicit expression for the  associated eigenstate of the present example is given by  \eqref{gg13}. The action of the noncommutative Ruelle operator obtained from  a  general  constant potential is described in  Subsection \ref{some1} (see also \eqref{ute}).
	
	 Consider a given line-stochastic matrix \(P \in M_{2 \times 2}(\mathbb{R})\), such that:
	\begin{equation*}
	P = \begin{bmatrix}
	p_{11} & p_{12} \\
	p_{21} & p_{22}
	\end{bmatrix} \quad \text{, where } p_{ij} > 0 \text.
	\end{equation*}
	
	Define:
	\begin{equation*}
	P_1 = \begin{bmatrix}
	\sqrt{p_{11}} & 0 \\
	0 & \sqrt{p_{21}}
	\end{bmatrix} \text, \quad
	P_2 = \begin{bmatrix}
	\sqrt{p_{12}} & 0 \\
	0 & \sqrt{p_{22}}
	\end{bmatrix} \text,
	\end{equation*}
	and:
	\begin{equation*}
	K = 
	\begin{bmatrix}
	1 & 0 \\
	0 & 0
	\end{bmatrix} \otimes P_1
	+
	\begin{bmatrix}
	0 & 0 \\
	0 & 1
	\end{bmatrix} \otimes P_2 \text.
	\end{equation*}
	Notice that \(K = K^{*}\).
	
	Let
	\begin{equation*}
	x \in \Omega\to a(x) = \begin{bmatrix}
	a_{11}(x) & a_{12} (x)\\
	a_{21}(x) & a_{22}(x)
	\end{bmatrix}\in   M_{2 \times 2}(\mathbb{R}),
	\end{equation*}
	 and denote by \( \operatorname{Id} \) the identity matrix in \( M_{2 \times 2}(\mathbb{R}) \).  We will {\bf omit the dependence on $x\in \Omega $} in the next computations.
	We wish to consider a potential of the form \(a \mapsto \varepsilon(a \otimes \operatorname{Id})\), where:
	\begin{equation*}
	\varepsilon(a \otimes \operatorname{Id}) := \operatorname{Tr}_1 \left( K^{*} (a \otimes \operatorname{Id}) K \right) 
  \end{equation*}
  \begin{equation*}
	= \operatorname{Tr}_1 \left(
    \left( \begin{array}{l}
	\begin{bmatrix}
	1 & 0 \\
	0 & 0
	\end{bmatrix} \otimes P_1 + \\ +
	\begin{bmatrix}
	0 & 0 \\
	0 & 1
	\end{bmatrix} \otimes P_2
  \end{array} \right)
	(a \otimes \operatorname{Id})
  \left( \begin{array}{l}
	\begin{bmatrix}
	1 & 0 \\
	0 & 0
	\end{bmatrix} \otimes P_1 + \\ +
	\begin{bmatrix}
	0 & 0 \\
	0 & 1
	\end{bmatrix} \otimes P_2
  \end{array} \right)
	\right)
  \end{equation*}
  \begin{equation*}
	= \operatorname{Tr}_1 \left(
	\begin{bmatrix}
	1 & 0 \\
	0 & 0
	\end{bmatrix} a
	\begin{bmatrix}
	1 & 0 \\
	0 & 0
	\end{bmatrix} \otimes P_1^2
	+
	\begin{bmatrix}
	0 & 0 \\
	0 & 1
	\end{bmatrix} a
	\begin{bmatrix}
	0 & 0 \\
	0 & 1
	\end{bmatrix} \otimes P_2^2
	\right) 
  \end{equation*}
  \begin{equation}\label{gg9}
	= \operatorname{Tr}_1 \left(
	\begin{bmatrix}
	a_{11} & 0 \\
	0 & 0
	\end{bmatrix} \otimes P_1^2
	+
	\begin{bmatrix}
	0 & 0 \\
	0 & a_{22}
	\end{bmatrix} \otimes P_2^2
	\right) 
	= a_{11} P_1^2 + a_{22} P_2^2 \text.
	\end{equation}
	If \(a\) is a positive matrix, then \(a_{11},\, a_{22} \geq 0\), and whenever \(a \neq 0\), at least one of them must
	be strictly positive. Since both \(P_{1}^{2}\) and \(P_{2}^{2}\) are strictly positive, the linear mapping 
	\(a \mapsto \varepsilon(a \otimes \operatorname{Id})\) is positivity-improving, for each $x \in \Omega$. Also, since \(P\) was taken
	line-stochastic, the linear mapping \(a \mapsto \varepsilon(a \otimes \operatorname{Id})\) is unital.
	
	There exists a unique  vector of probability $\pi=(\pi_1,\pi_2)$ such that $\pi P =\pi.$ In this way,
	\begin{equation} \label{or1} \pi_1\, p_{11} + \pi_2 \,p_{21} = \pi_1, \end{equation}
	and
	\begin{equation} \label{or2}  \pi_1\, p_{12} + \pi_2 \,p_{22} = \pi_2.\end{equation}

	 In order to be consistent with the previous notation of Theorem \ref{thm:ruelle-normal}, applied to the present case, note  that the Ruelle operator acts on functions $g:\Omega\to  M_{2 \times 2}$, and by abuse of notation  let's write $g(a)(x) = a(x).$ Then, the  noncommutative Ruelle operator is 
	  \begin{equation} \label{gg10}L_\varphi (g)(x)=g_{11}(1 x) P_1^2 + g_{22} (2 x) P_2^2.
	  \end{equation}
	We want to find the eigenstate \(\eta\) for the Ruelle operator $L_\varphi$.  
	 First, we consider the state $\xi: M_{2 \times 2} \to \mathbb{R}$ such that when applied to a matrix $a(x).$
	 \begin{equation} \label{gg8}\xi (a(x)) = a_{11}(x) \pi_1 + a_{22}(x) \pi_2.
	 \end{equation}

$\xi: M_{2 \times 2} \to \mathbb{R}$ is the eigenstate for $a\in M_{2 \times 2}  \to \varepsilon(a \otimes \operatorname{Id})$. Indeed, for any matrix $a\in M_{2 \times 2} $, from \eqref{or1} and \eqref{or2}
	$$ \xi (\,  \varepsilon(a \otimes \operatorname{Id})\,)=\xi( \, a_{11} P_1^2 + a_{22} P_2^2\,)= \pi_1 ( a_{11} p_{11} +a_{22} p_{12}) + \pi_2 ( a_{11} p_{21} +a_{22} p_{22}  )=$$
	$$=a_{11}  (   \pi_1\, p_{11} + \pi_2 \,p_{21})\, + a_{22} (\pi_1\, p_{12} + \pi_2 \,p_{22}  )=  a_{11} \pi_1 + a_{22} \pi_2=\xi (a).$$
	
	It follows that for any $x$,
	\begin{equation} \label{gg11} \xi (\,  \varepsilon(a (x)\otimes \operatorname{Id})\,)=  a_{11}(x) \pi_1 + a_{22}(x)
	 \pi_2=\xi (a(x)).\end{equation}
	 
	  In order to be consistent the previous notation of Theorem \ref{thm:ruelle-normal}, we write the above expression as
\begin{equation} \label{gg12} \xi (\,  \varepsilon(g\otimes \operatorname{Id})\,)=  g_{11}(x) \pi_1 + g_{22}(x)
	 \pi_2=\xi (a(x)).\end{equation}
	 
	 Now we want to explicit  the eigenstate \(\eta\) for the Ruelle operator $L_\varphi$.
From Subsection \ref{some1} it follows that the eigenstate \(\eta\) for the transfer operator associated with the constant
  potential \(g \mapsto \varepsilon(g \otimes \operatorname{Id})\) is given by simply integrating \(\xi\) against the
  maximal entropy measure $\mu$; if \(g \in \mathcal{C}(\Omega; M_{2 \times 2}(\mathbb{R}))\):
  \begin{equation} \label{gg13}
    \eta(g) = \int (g_{11}(x) \,\pi_{1} + g_{22}(x) \,\pi_{2})\, \operatorname{d}\! \mu (x)\text.
  \end{equation}

\end{example}

\medskip

\section{Proof of Theorem \ref{thm:ruelle-normal}} \label{mma}

It will be of fundamental importance in our reasoning the use of two results, which are stated below for the
convenience of the reader. First, the Ionescu-Tulcea and Marinescu Theorem, as given in \cite[Theorem B.1.1]{finertherm}.

\begin{theorem*}[Ionescu-Tulcea and Marinescu]
	Let $E$ be a complex linear (i.e., vector) space and $B$ a linear subspace of $E$. Suppose that $E$ and $B$ are 
	Banach spaces when equipped with the norms $|\cdot|$ and $\|\cdot\|$, respectively. Suppose also that:
	\begin{itemize}
		\item[(a1)] If a sequence $(x_n)_{n=1}^{\infty}$ in $B$ is such that there exists a constant $K > 0$ with:
		\begin{equation*}
		\lim_{n \to \infty} |x_n - x| = 0 \quad \text{for some } x \in E
		\quad \text{and} \quad \|x_n\| \leq K \text{ for all } n \in \mathbb{N},
		\end{equation*}
		then:
		\begin{equation*}
		x \in B \quad \text{and} \quad \|x\| \leq K.
		\end{equation*}
	\end{itemize}
	
	Furthermore, let $T : B \to B$ be a bounded linear operator satisfying the following conditions:
	\begin{itemize}
		\item[(b1)] There exists a constant $H > 0$ such that $|T^n| \leq H$ for all $n \in \mathbb{N}$.
		
		\item[(c1)] There exist $N \in \mathbb{N}$ and constants $0 < r < 1$ and $R > 0$ such that:
		\begin{equation*}
		|T^N(x)| \leq r \|x\| + R |x|, \quad \forall x \in B.
		\end{equation*}
		
		\item[(d1)] Every bounded subset of $B$ is mapped by $T$ to a relatively compact subset of $E$.
		That is, the operator $T : (B, \|\cdot\|) \to (E, |\cdot|)$ is compact.
	\end{itemize}
	
	Then:
	\begin{itemize}
		\item[(a2)] There are only a finite number of eigenvalues of $T : B \to B$ with modulus equal to $1$. We denote them by $\lambda_1, \ldots, \lambda_p \in \mathbb{S}^1 \subseteq \mathbb{C}$.
		
		\item[(b2)] There exist operators $T_1, \ldots, T_p, S$, all satisfying properties (b1), (c1), and (d1), such that:
		\begin{equation*}
		T^n = \sum_{i=1}^{p} \lambda_i^n T_i + S^n, \quad \forall n \in \mathbb{N},
		\end{equation*}
		where $\dim(T_i(B)) < \infty$ for every $1 \leq i \leq p$.
		
		\item[(c2)] $T_i^2 = T_i$, \quad $T_i \circ T_j = 0$ for $i \neq j$, \quad and \quad $T_i \circ S = S \circ T_i = 0$ \linebreak for each $1 \leq i,j \leq p$.
		
		\item[(d2)] $T \circ T_i = T_i \circ T = \lambda_i T_i$ for every $1 \leq i \leq p$.
		
		\item[(e2)] $r(S) < 1$, where $r(S)$ is the spectral radius of $S : B \to B$.
		
		\item[(f2)] Addendum: $\sigma(T) \setminus \{ \lambda_1, \ldots, \lambda_p \} \subset \sigma(S) \subset B(0, r(S)) \subset \mathbb{C}$, where $\sigma(T)$ denotes the spectrum of $T$. This property is called the \emph{spectral gap} of $T$.
	\end{itemize}
\end{theorem*}

Moreover, we will also need  Theorem 4.1 from \cite{glwe}.

\begin{theorem*}[Glück and Weber]
	Let $X$ be an ordered Banach space with positive cone $X_+ \neq \{0\}$,
	let $J = \mathbb{N}_0$ or $J = [0, \infty)$, and let $T = (T_t)_{t \in J}$ be a positive operator semigroup on
	$X$. Suppose that the following two assumptions are satisfied:
	\begin{itemize}
		\item[(a)] For each vector $0 \neq x \in X_+$, there exists a time $t_x \in J$ such that $T_{t_x}x$ is an almost interior point of $X_+$.
		
		\item[(b)] The semigroup $T$ is bounded and strongly asymptotically compact.
	\end{itemize}
	
	Then $T_t$ converges strongly as $t \to \infty$. If the limit operator $Q := \lim_{t \to \infty} T_t$ is not
	zero, then it is of the form $Q = y' \otimes y$. Here, $y$ is an almost interior point of $X_+$
	and a fixed point of $T$, and $y' \in X'$ is a strictly positive functional and a fixed point
	of the adjoint semigroup $T' := (T_t')_{t \in J}$ (consisting of the dual operators $T_t'$).
\end{theorem*}

\begin{remark}
	Recall that \(x \in X_+\) is called an \textit{almost interior point} if it is such that \(\mu(x) > 0\) for every
	non-zero positive functional \(\mu\) \cite[Definition 2.4 (ii)]{glwe}. Also, the operator semigroup \(T\) is said
	to be \textit{strongly asymptotically compact} if for each \(x \in X\) and each sequence $(t_n)_{n \in \mathbb{N}} \subseteq J$ 
	that converges to $\infty$, the sequence $(T_{t_n} x)_{n \in \mathbb{N}}$ in $X$ has a convergent subsequence 
	\cite[Definition 3.1]{glwe}. If convenient, the reader can  explore more details in reference \cite{glwe}  to gain 
	familiarity with these definitions and Theorem 4.1 therein.
\end{remark}

\begin{remark}
	We want to show that \cite[Theorem B.1.1]{finertherm} applies to \(E = \mathcal{C}(\Omega; \mathcal{A})\), \(B = \mathcal{C}_{\theta}(\Omega; \mathcal{A})\),
	and \(T = L_{\varphi}\). Also, that \cite[Theorem 4.1]{glwe} applies to \linebreak \(X = \mathcal{C}(\Omega; \mathcal{A})\), 
	\(X_{+} = \mathcal{C}(\Omega; \mathcal{A})^{+}\), and \(T_{t} = L_{\varphi}^{n}\).
\end{remark}

For the proof of our main result, let us begin by proving hypothesis (b1) of \cite[Theorem B.1.1]{finertherm}. Since each
\(\varphi_{ix}\) is positivity-improving, and particularly positive, then so is their sum. But beware: this only
ensures \(\sum_{i} \varphi_{ix}\) to be positivity-improving for every \(x \in \Omega_{A}\), which is not enough to
ensure the linear operator \(L_{\varphi} : \mathcal{C}(\Omega_{A}; \mathcal{A}) \to \mathcal{C}(\Omega_{A}; \mathcal{A})\)
to be positivity-improving. 

We will see later that \(L_{\varphi}\) is merely ``eventually'' positivity-improving,
in a sense to be made precise then; this will be sufficient. For now, it is enough to know that each \(\varphi_{x}\) being positive for every \(x
\in \Omega_{A}\), the linear operator \(L_{\varphi}\) follows to be positive. Therefore, we may employ \cite[Corollary 2.9 (Russo-Dye)]{paul}
to conclude that:
\begin{equation*}
\norm{L_{\varphi}} = \abs{L_{\varphi}(\mathbb{I}_{\mathcal{C}})}_{\infty} = \sup_{x \in \Omega_{A}} \abs{\mathbb{I}_{\mathcal{A}}}_{\mathcal{A}} = 1 \text.
\end{equation*}

Now, we pass to the issue of showing \(L_{\varphi}\) is ``weakly contracting for Lipschitz functions'', in the sense that inequality 
\eqref{eq:itineq} holds. This corresponds to showing \(L_{\varphi}\) satisfies hypothesis (c1) of \cite[Theorem B.1.1]{finertherm}. 
To do so, first, we prove the basic inequality \cite[Proposition 2.1]{papo}. We consider two cases:

\bigskip \noindent \textbf{Case 1.} \textit{Say \(d(x,y) > \theta\). If this holds, then:
	\begin{equation*}
	\abs{(L_{\varphi} g)(x) - (L_{\varphi} g)(y)}_{\mathcal{A}} \leq 2 \norm{L_{\varphi}} \abs{g}_{\infty} \leq
	2 \abs{g}_{\infty} \frac{d(x,y)}{\theta} \text.
	\end{equation*}}

\bigskip \noindent \textbf{Case 2.} \textit{Say \(d(x,y) \leq \theta\). If this holds instead, then \(x_{1} = y_{1}\), so that:
	\begin{equation*}
	\abs{(L_{\varphi} g)(x) - (L_{\varphi} g)(y)}_{\mathcal{A}} \leq \sum_{\substack{i \in \Sigma \\ A(i,x_{1}) = 1}}
	\abs{\varphi_{ix} \big( g(ix) \big) - \varphi_{iy} \big( g(iy) \big)}_{\mathcal{A}} 
  \end{equation*}
	\begin{equation*}
	\leq \sum_{i} \norm{\varphi_{ix} - \varphi_{iy}} \abs{g(ix)}_{\mathcal{A}}
	+ \abs{\sum_{i} \varphi_{iy} \left( g(ix) - g(iy) \right)}_{\mathcal{A}} 
  \end{equation*}
	\begin{equation*}
	\leq \sum_{i} \norm{\varphi_{ix} - \varphi_{iy}} \abs{g(ix)}_{\mathcal{A}}
	+ \abs{\sum_{i} \varphi_{iy} \left( \abs{g(ix) - g(iy)}_{\mathcal{A}} \mathbb{I}_{\mathcal{A}} \right)}_{\mathcal{A}} 
  \end{equation*}
	\begin{equation*}
	\leq \sum_{i} \norm{\varphi_{ix} - \varphi_{iy}} \abs{g(ix)}_{\mathcal{A}}
	+ \abs{\sum_{i} \varphi_{iy} \left( d(ix, iy) \abs{g}_{\theta} \mathbb{I}_{\mathcal{A}} \right)}_{\mathcal{A}} 
  \end{equation*}
	\begin{equation*}
	\leq k \abs{\varphi}_{\theta} \theta d(x, y) \abs{g}_{\infty} + \theta d(x, y) \abs{g}_{\theta} \text.
	\end{equation*}} 

Define \(C_{1} := C_{1}(\varphi, \theta) := \max \left\{k \theta \abs{\varphi}_{\theta} , \frac{2}{\theta}\right\}\). 
Combining both cases, it follows that:
\begin{equation*}
\abs{L_{\varphi} g}_{\theta} \leq C_{1} \abs{g}_{\infty} + \theta \abs{g}_{\theta} \text.
\end{equation*}

Proceeding by induction, assume \(\abs{L_{\varphi}^{n} g}_{\theta} \leq C_{n} \abs{g}_{\infty} + \theta^{n} \abs{g}_{\theta}\).
Then:
\begin{equation*}
\abs{L_{\varphi}^{n+1} g}_{\theta} = \abs{L_{\varphi}^{n} \left( L_{\varphi} g\right)}_{\theta} \leq 
C_{n} \abs{L_{\varphi} g}_{\infty} + \theta^{n} \abs{L_{\varphi} g}_{\theta} 
\end{equation*}
\begin{equation*}
 \leq
C_{n} \abs{L_{\varphi} g}_{\infty} +  \theta^{n} \left( C_{1} \abs{g}_{\infty} + \theta \abs{g}_{\theta} \right) 
\end{equation*}
\begin{equation*}
\leq
\left( C_{n} \norm{L_{\varphi}} + \theta^{n} C_{1} \right) \abs{g}_{\infty} + \theta^{n+1} \abs{g}_{\theta} 
\end{equation*}
\begin{equation*}
\leq
C_{n+1} \abs{g}_{\infty} + \theta^{n+1} \abs{g}_{\theta} \text.
\end{equation*}

Since \(\theta < 1\), we may recursively choose \(C_{n+1}\) as:
\begin{equation*}
C_{n+1} := C_{n} + \theta^{n} C_{1} = \sum_{m = 0}^{n} \theta^{m} C_{1}  
= \frac{1 - \theta^{n+1}}{1 - \theta} C_{1} 
\end{equation*}
\begin{equation*}
\leq \frac{C_{1}}{1 - \theta} := C(\varphi, \theta) := C \text.
\end{equation*}

Therefore:
\begin{equation} \label{eq:basic}
\abs{L_{\varphi}^{n} g}_{\theta} \leq C \abs{g}_{\infty} + \theta^{n} \abs{g}_{\theta} \text.
\end{equation}

Second, we infer the Ionescu-Tulcea and Marinescu Inequality \cite[Corollary 13.8.16]{finertherm} by observing that:
\begin{equation*} 
\norm{L_{\varphi}^{n} g}_{\theta} := \abs{L_{\varphi}^{n} g}_{\infty} + \abs{L_{\varphi}^{n} g}_{\theta} 
\leq \norm{L_{\varphi}^{n}} \abs{g}_{\infty} + C \abs{g}_{\infty} + \theta^{n} \abs{g}_{\theta} 
\end{equation*}
\begin{equation} \label{eq:itineq}
\leq \left(C + 1\right) \abs{g}_{\infty} + \theta^{n} \norm{g}_{\theta} \text.
\end{equation}

The hypothesis (d1) of \cite[Theorem B.1.1]{finertherm} follows from \cite[Remark B.1.2]{finertherm} plus Remark \ref{re:arz-asc}.

Finally, the hypothesis (a1) of \cite[Theorem B.1.1]{finertherm} can be verified directly: let
\((g_{n})_{n \in \mathbb{N}}\) be a sequence of Lipschitz-continuous functions \(g_{n} \in \mathcal{C}_{\theta}(\Omega_{A}; \mathcal{A})\)
such that \(\norm{g_{n}}_{\theta} \leq K\) for every \(n \in \mathbb{N}\) and some \(K \in \mathbb{R}\). Suppose \(g_{n}\)
converges uniformly to \(g\), that is:
\begin{equation*}
\lim_{n \to +\infty} \abs{g_{n} - g}_{\infty} = 0 \text.
\end{equation*}

Observe that \(g \in \mathcal{C}_{\theta}(\Omega_{A}; \mathcal{A})\) and that the Lipschitz norm of \(g\) is bounded above by
\(K\), because:
\begin{equation*}
\norm{g}_{\theta} := \abs{g}_{\infty} + \abs{g}_{\theta} = \lim_{n \to +\infty} \left( \abs{g_{n}}_{\infty} \right)
+ \sup_{x \neq y \in \Omega_{A}} \frac{\abs{g(x) - g(y)}_{\mathcal{A}}}{d_{\theta}(x,y)} 
\end{equation*}
\begin{equation*}
\leq \lim_{n \to +\infty} \left( \abs{g_{n}}_{\infty} \right)
+ \sup_{x \neq y \in \Omega_{A}} \liminf_{n \to +\infty} \left( \frac{\abs{g_{n}(x) - g_{n}(y)}_{\mathcal{A}}}{d_{\theta}(x,y)} \right) 
\end{equation*}
\begin{equation*}
\leq \lim_{n \to +\infty} \left( \abs{g_{n}}_{\infty} \right)
+ \liminf_{n \to +\infty} \left( \abs{g_{n}}_{\theta} \right) 
= \liminf_{n \to +\infty} \left( \abs{g_{n}}_{\infty} + \abs{g_{n}}_{\theta} \right) 
\end{equation*}
\begin{equation*}
= \liminf_{n \to +\infty} \norm{g_{n}}_{\theta} 
\leq K \text.
\end{equation*}

This means we are all set to apply \cite[Theorem B.1.1]{finertherm}. So, we already know \(1 \in \mathbb{R}\) is an isolated eigenvalue
of \(L_{\varphi}\). To see that it is simple, suppose there is some \(g \in \mathcal{C}(\Omega_{A}; \mathcal{A})\) such that
\(L_{\varphi} g = g\), that is, \(g \neq 0\) is an eigenfunction corresponding to the eigenvalue \(1\). By uniquely
decomposing each \(a \in \mathcal{A}\) (and, simultaneously, each \(g \in \mathcal{C}(\Omega_{A}; \mathcal{A})\)) into positive elements 
(and, simultaneously, into positive functions) as per \cite[Lemma 4.38]{bp-qsm}, and noticing that \(L_{\varphi}\) 
leaves the set of positive functions invariant, we may assume, without loss of generality, that \(g \geq 0\).

Being a fixed point of \(L_{\varphi}\), we conclude that \(g > 0\). This happens because:
\begin{align*}
& & g(x) = (L_{\varphi}g)(x) = \sum_{i} \varphi_{ix} g(ix) \text{, and:} \\ 
& & g(x) \geq 0 \text, \quad \forall x \in \Omega \\
& \implies & g(ix) \geq 0 \text, \quad \forall ix \in \Omega \\ 
& \implies & \left( \varphi_{ix} g(ix) > 0 \text{, or: } g(ix) = 0 \right) \quad \forall ix \in \Omega \text,
\end{align*}
and the possibility that \(g(ix) = 0\) for any (some) \(ix \in \Omega\) can be ruled out because it implies \(g(y) = 0\) 
for every \(y \in \cup_{n \in \mathbb{N}} \sigma^{-n}(ix)\), and by continuity, for every \(x \in \Omega_{A}\);
contradicting \(g \neq 0\). It follows that \(\varphi_{ix} g(ix) > 0\) for every \(ix \in \Omega\), and, adding
these, that \(g > 0\).

Consider the finite constant \(\mathbb{R} \ni s > 0\) given by:
\begin{equation*}
s := \sup \left\{t \in \mathbb{R} \, \middle| \, t \mathbb{I}_{\mathcal{C}} \leq g\right\} \text.
\end{equation*}
Notice that \(s\) is bounded above, for instance, by any finite upper bound for \(g\), and bounded below, for instance, 
by any positive lower bound for \(g\). Notice also that \(g - s \mathbb{I}_{\mathcal{C}} \geq 0\), but \textbf{not} \(g - s \mathbb{I}_{\mathcal{C}} > 0\). 
But, being a linear combination of two fixed points of \(L_{\varphi}\), \(g - s \mathbb{I}_{\mathcal{C}}\), it is also a fixed point
of \(L_{\varphi}\). We have just seen that any non-zero positive fixed points of \(L_{\varphi}\) must be 
strictly positive. Thus, \(g - s \mathbb{I}_{\mathcal{C}}\) must be zero, that is, \(g = s \mathbb{I}_{\mathcal{C}}\).

Regarding the dual space \(\mathcal{C}(\Omega_{A}; \mathcal{A})^{\prime}\), and the corresponding dual operator
\(L_{\varphi}^{\prime}\), we remark that the set:
\begin{equation*}
\mathcal{S}\left(\mathcal{C}(\Omega_{A}; \mathcal{A})\right) := \left\{\eta \in \mathcal{C}(\Omega_{A}; \mathcal{A})^{\prime} \middle| \norm{\eta} = \eta(\mathbb{I}_{\mathcal{C}}) = 1\right\} \text,
\end{equation*}
and its subset:
\begin{equation*}
\mathcal{S}_{\sigma}(\mathcal{C}(\Omega_{A}; \mathcal{A})) := \left\{\eta \in \mathcal{S}\left(\mathcal{C}(\Omega_{A}; \mathcal{A})\right) \middle| \sigma_{\sharp}(\eta) = \eta\right\} \text,
\end{equation*}
where \(\sigma_{\sharp}\) denotes the push-forward mapping, defined by:
\begin{equation*}
(\sigma_{\sharp}\eta)(g) := \eta (g \circ \sigma) \text,
\end{equation*}
are both weak*-compact and invariant by \(L_{\varphi}^{\prime}\). Therefore, there is a \(\sigma\)-invariant state
\(\eta_{\varphi}\) such that \(L_{\varphi}^{\prime}(\eta_{\varphi}) = \eta_{\varphi}\). We now see that it 
is unique and characterized by item \ref{thmitm:unilim} of Theorem \ref{thm:ruelle-normal} by employing \cite[Theorem 4.1]{glwe}
as follows:

The basic inequality \eqref{eq:basic} shows the sequence \((L_{\varphi}^{n}(g))_{n \in \mathbb{N}}\) to be 
equicontinuous for every \(g \in \mathcal{C}(\Omega_{A}; \mathcal{A})\). Since \(\norm{L_{\varphi}} = 1\), the sequence 
\((L_{\varphi}^{n}(g))_{n \in \mathbb{N}}\) is also bounded for every \(g \in \mathcal{C}(\Omega_{A}; \mathcal{A})\). 
Combining all  these with the Arzelà-Ascoli Theorem, it follows that hypothesis (b) of \cite[Theorem 4.1]{glwe} is satisfied.

Now say that \(g \in \mathcal{C}(\Omega_{A}; \mathcal{A})^{+}\) and \(g \neq 0\). This means there is some \(x \in \Omega_{A}\)
such that \(g(x) \neq 0\). By continuity, there is an open set \(U \subseteq \Omega_{A}\) for which \(0 \notin g(U)\). 
Since \(A\) is aperiodic, there exists an \(N = N(U) \in \mathbb{N}\) such that \(\sigma^{-n}(x) \cap U \neq \emptyset\),
for any \(x \in \Omega_{A}\), and \(n \geq N\).
Then, \(L_{\varphi}^{n}(g) > 0\), since, for any \(x \in \Omega_{A}\):
\begin{equation*}
(L_{\varphi}^{n}g)(x) = \sum_{y \in \sigma^{-n}(x)}  \left(\prod_{i = 1}^{n} \varphi_{\sigma^{i}(y)}\right) g(y) \text,
\end{equation*}
and \(g(y) \neq 0\) for at least one \(y \in \sigma^{-n}(x)\). This shows that hypothesis (a) of \cite[Theorem 4.1]{glwe}
is also satisfied. It follows that \(L_{\varphi}^{n}(g)\) converges uniformly to \(\eta_{\varphi}(g)
\mathbb{I}_{\mathcal{C}}\), and also that \(\eta_{\varphi}\) is a faithful state \cite[Definition 4.61]{bp-qsm}, that is, strictly positive \cite[Definition 2.4]{glwe}.

To see that there is no other unitary eigenvalue of \(L_{\varphi}\), suppose the contrary: that there is 
\(g \in \mathcal{C}(\Omega_{A}; \mathcal{A})\) such that \(g \neq 0\), \(L_{\varphi} g = \lambda g\), \(\lambda \in \mathbb{R}\), 
\(\abs{\lambda} = 1\), \(\lambda \neq 1\). Notice: 
\begin{equation} \label{eq:etag-zero}
\eta_{\varphi}(g) = \left(L_{\varphi}^{\prime} \eta_{\varphi}\right)(g)
= \eta_{\varphi}(L_{\varphi} g)
= \lambda \eta_{\varphi}(g) \text,
\end{equation}
so that \(\eta_{\varphi}(g) = 0\). It follows that \(\lim_{n} L_{\varphi}^{n} g = \lambda^{n} g \equiv 0\),
which, together with \(\abs{\lambda} = 1\), implies that \(g = 0\), a contradiction.

Finally, to see that any positive eigenfunction of \(L_{\varphi}\) must be a positive multiple of \(\mathbb{I}_{\mathcal{C}}\),
suppose the contrary: that there is \(g \in \mathcal{C}(\Omega_{A}; \mathcal{A})^{+}\) such that \(g \neq 0\), \(L_{\varphi} g = \lambda g\), 
\(\lambda \in \mathbb{R}\), \(\abs{\lambda} < 1\). Repeating the reasoning of \eqref{eq:etag-zero}, we again conclude 
that \(\eta_{\varphi}(g) = 0\). Since \(\eta_{\varphi}\) is faithful (strictly positive) and \(g\) is positive, it
follows that \(g = 0\), a contradiction.

\section{Entropy for a certain family of potentials} \label{Pre}

In the classical thermodynamic formalism, the entropy \(h_{\mu}(\sigma)\) of an equilibrium state \(\mu\) for a 
normalized potential \(B : \Omega \to \mathbb{R}\) admits the expression:
\begin{equation*}
h_{\mu}(\sigma) = - \int \log J_{\mu} \operatorname{d}\! \mu = - \int B \operatorname{d}\! \mu \text,
\end{equation*}
where \(J_{\mu}: \Omega \to (0,1)\) denotes the so called Jacobian of \(\mu\) with respect to the shift dynamics \(\sigma\) (see \cite{olivi-fte}, \cite{craizer-m} or Section 3 in \cite{lopes-nft}). This follows from 
the variational principle of pressure and the normalization condition, and \(J_{\mu}\) can be interpreted as the point-wise expansion 
factor of the measure (or as a kind of Radon-Nikodym derivative in inverse branches, as described in \cite{craizer-m,} or in  Section 7 in \cite{lopes-nft}).

\begin{definition}
	In our \(\mathcal{A}\)-valued setting, for a normalized potential \(\varphi \in \mathcal{C}(\Omega_{A}; \mathfrak{L}(\mathcal{A}))\), 
	it is natural to define the point-wise Jacobian $J$ as the \(\mathcal{A}\)-valued function:
	\begin{equation*}
	J(x) := \varphi_{x}(\mathbb{I}_{\mathcal{A}}) \text.
	\end{equation*}
	It is a generalization of the previuous concept of Jacobian.
\end{definition}

\begin{remark}
	Because of \cite[Corollary 2.9 (Russo-Dye)]{paul}, this definition implies the equality \(\abs{J(x)} := \norm{\varphi_{x}}\).

\end{remark}

\begin{definition} \label{efen}
	Given an eigenstate \(\eta\) of the transfer operator \(L_{\varphi}\), we propose the following ansatz for the entropy of \(\eta\):
	\begin{equation*}
	h(\varphi,\eta,\sigma) := - \eta(\log J) \text.
	\end{equation*}
\end{definition}

\begin{remark}
	Since \(\varphi_{x}\) is positivity-improving, the logarithm expression \(\log J\) is point-wise well defined. The property of it 
  being continuous follows from the continuity of \(\varphi : \Omega_{A} \to \mathfrak{L}(\mathcal{A})\), plus the compactness of
	\(\Omega_{A}\). Therefore, it makes sense to evaluate \(\eta\) at \(\log J\).
\end{remark}

\begin{remark} 
  In the classical (scalar) setting, the Kolmogorov-Sinai entropy is a functional of the invariant measure alone. In 
  our noncommutative framework the quantity \(h(\varphi,\eta,\sigma)\) depends on \emph{both} the eigenstate \(\eta\) 
  and the choice of normalized \(\mathfrak{L}(\mathcal A)\)-valued potential \(\varphi\). Distinct normalized potentials 
  may share the same \(\eta\) but yield different \(J_{\varphi}\) and hence different entropies. 
\end{remark}

\begin{example}
  Let \(\Omega_{A} = \left\{1,2\right\}^{\mathbb{N}}\), $ \mathcal{A} = M_{2 \times 2} $, \(\mathfrak{L}(\mathcal{A}) = \mathfrak{L}(M_{2 \times 2}(\mathbb{R}))\), and consider potentials
  \(\varphi_x(g) = (\widehat{\operatorname{tr}}\,g)P(x)\), where \(\widehat{\operatorname{tr}} = \frac{1}{2} \operatorname{Tr}\).
  Normalization in this situation means \(P(1y)+P(2y)=\operatorname{Id}\) for all \(y \in \left\{1,2\right\}^{\mathbb{N}}\).
  Note that \(\widehat{\operatorname{tr}}\,(\operatorname{Id}) = 1\), so \(J_\varphi(x) = \varphi_x(\operatorname{Id})=P(x)\).

  As explained in Example \ref{ex:tr-type}, the associated eigenstate for the transfer operator is \(\eta\) with
  \(\eta(g)=\int \widehat{\operatorname{tr}}\,(g(x)) \operatorname{d} \! \mu(x)\), where \(\mu\) denotes the equilibrium measure for the scalar 
  potential \(x \mapsto \widehat{\operatorname{tr}}\,P(x)\). In the next two examples, \(\widehat{\operatorname{tr}}\,P(\cdot) \equiv
  \frac{1}{2}\), so \(\mu\) is the Bernoulli \((\frac{1}{2},\frac{1}{2})\) measure.

  \begin{enumerate}[(i)]
    \item Constant case.

  Take \(P(x) \equiv \begin{bmatrix} \frac{1}{2} & 0 \\ 0 & \frac{1}{2} \end{bmatrix}\).
  Then \(\log J_\varphi=\log P=\begin{bmatrix} \log \frac{1}{2} & 0 \\ 0 & \log \frac{1}{2} \end{bmatrix}\), hence:
  \[
  h(\varphi,\eta,\sigma)
  = -\eta(\log J_\varphi)
  = -\int \widehat{\operatorname{tr}}\,(\log P) \operatorname{d} \! \mu
  = -\log\frac{1}{2}
  =\log 2.
  \]

  \item First-coordinate dependent case.

  Let \(p\in(0,1)\). Take:
  \[
  P(1y) = \begin{bmatrix} p & 0 \\ 0 & 1-p \end{bmatrix} \text,\qquad
  P(2y) = \begin{bmatrix} 1-p & 0 \\ 0 & p \end{bmatrix} \text,
  \]
  so that \(P(1y)+P(2y)=\operatorname{Id}\).
  Then, for either the first symbol:
  \[
  \widehat{\operatorname{tr}}\,(\log P(i y)) = \frac{1}{2} \big(\log p+\log(1-p)\big) \text,
  \]
  and therefore:
  \[
  h(\varphi,\eta,\sigma)
  = -\int (\widehat{\operatorname{tr}}\,\log P(i y)) \operatorname{d} \! \mu
  \]
  \[
  = -\frac{1}{2} \big(\log p+\log(1-p)\big)
  = -\frac{1}{2} \log\big(p(1-p)\big) \text.
  \]
  This quantity is minimized at \(p=\frac{1}{2}\), where \(h(\varphi,\eta,\sigma) = \log 2\), and:
  \[
  h(\varphi,\eta,\sigma) \longrightarrow +\infty
  \text{ as } p \to 0 \text{ or } p \to 1 \text.
  \]
  \end{enumerate}

  In particular, both potentials yield the same eigenstate \(\eta\) but produce different entropies \(h(\varphi,\eta,\sigma)\), 
  illustrating that our entropy is a function of the pair \((\eta,\varphi)\), not of \(\eta\) alone. Admittedly, \(\eta\)
  depends on \(\varphi\), and in that sense, this quantity we are labeling as entropy, would be a function of the
  potential.
\end{example}

Now, in order to  justify more properly our ansatz, consider the case of Example \ref{ex:tr-type}. \(\Omega_{A}\) is an aperiodic subshift of
finite type, \(\mathcal{A} = M_{d \times d}(\mathbb{R})\), and the potential \(\varphi : \Omega_{A} \to \mathfrak{L}(\mathcal{A})\)
is characterized by 
\begin{equation}\label{prec3}\varphi_{x}(g(x)) = \left( \widehat{\operatorname{tr}}\, g(x) \right)P(x),
\end{equation} where \(P(x)\) is
a (Lipshitz  continuous) family of strictly positive matrices such that:
\begin{equation}\label{prec}
\sum_{\substack{i \in \Sigma \\ A(i,x_{1}) = 1}} P(ix) = \operatorname{Id} \text,
\end{equation}
for all \(x \in \Omega_{A}\). 

\begin{definition} \label{def:tr-type}
	In case \(\varphi\) is defined via the previous paragraph (see \eqref{prec3} and \eqref{prec}), we say it is of \emph{trace-type} with \emph{matrix-factor} \(P\).
\end{definition}

We have seen that when \(\varphi\) is of trace-type, the eigenstate \(\eta\) for the \(\mathcal{A}\)-valued transfer operator \(L_{\varphi}\)
is given by integrating the normalized trace against the eigenprobability  $
\mu_{\widehat{\operatorname{tr}}\, P} $,  for the classical transfer
operator \(L_{\widehat{\operatorname{tr}}\, P}\).

Note that $\eta$ is invariant for the action of the shift $\sigma$ in the sense that for any $g$
we get
\begin{equation}\label{prec1}\eta(g \circ \sigma)= \int  \widehat{\operatorname{tr}} (g \circ \sigma) d \mu_{\widehat{\operatorname{tr}}\, P} = \int  \widehat{\operatorname{tr}} (g) d \mu_{\widehat{\operatorname{tr}}\, P}= \eta(g).
\end{equation}

Moreover, we claim that  the entropy of such an eigenstate \(\eta\)
is  related to the one for  \(\mu\). In fact, by employing our definition of entropy, we have the following inequality:
\begin{equation*}
h(\varphi,\eta,\sigma) := - \eta(\log J) = - \int \widehat{\operatorname{tr}}\, \left( \log \varphi_{x}(\mathbb{I}_{\mathcal{A}}) \right) \operatorname{d} \!
\mu_{\widehat{\operatorname{tr}}\, P}(x) =
\end{equation*}
\begin{equation*}
= - \int \widehat{\operatorname{tr}}\, \left( \log P(x) \right) \operatorname{d} \!
\mu_{\widehat{\operatorname{tr}}\, P}(x)
\geq - \int \log \left( \widehat{\operatorname{tr}}\, P(x) \right) \operatorname{d} \!
\mu_{\widehat{\operatorname{tr}}\, P}(x)
= h_{\mu_{\widehat{\operatorname{tr}}\, P}}(\sigma) \text,
\end{equation*}
where \(h_{\mu_{\widehat{\operatorname{tr}}\, P}}(\sigma)\) denotes the Kolmogorov-Sinai entropy of \(\mu_{\widehat{\operatorname{tr}}\, P}\)
with respect to \(\sigma\). If we let \(d = 1\), we are back to the classical setting. Then the trace becomes trivial, 
and the inequality becomes an equality, as expected.
\medskip

We believe this is a satisfactory justification for the definition of entropy introduced above.

\medskip
Our definition of entropy also satisfies an inequality that is somehow  associated with the variational principle in the case the pressure (as defined in \cite{papo}) is zero.
More precisely, if $J: \Omega \to (0,1)$ is normalized and Holder
we get that
$$\sup_{\tilde{\mu} \, \text{ invariant for}\, \sigma} \{\,h(\tilde{\mu}) + \int \log J\, d \tilde{\mu} \} =0.$$ 
The maximal value $0$ is attained when $\tilde{\mu}= \mu_{J}$, where $\mathcal{L}_J^* ( \mu_{J} ) =\mu_{J} $.

\medskip

Here we can show the following:
\begin{theorem} \label{prep}
	
  Let \(\varphi\) be a normalized trace-type potential, with eigenstate \(\eta\) and jacobian \(J\). Then,  
  \begin{equation} \label{ineq:norm-var}  h(\psi,\tilde{\eta},\sigma) + \tilde{\eta}(\log J) \leq 0 \text,
	\end{equation}
	where \(\tilde{\eta}\) is the eigenstate for a normalized trace-type potential \(\psi\). 
  The value $0$ in \eqref{ineq:norm-var} is obtained when \(\tilde{\eta} = \eta\) and $\psi = \varphi$.
\end{theorem}

\begin{proof}
	By definition, \(h(\varphi,\eta,\sigma) + \eta(\log J) = 0\). Furthermore, if \(\tilde{J}\) denotes the Jacobian for \(\psi\), we have, 
	also by definition, \(h(\psi,\tilde{\eta},\sigma) + \tilde{\eta}(\log \tilde{J}) = 0\).
	Therefore, \eqref{ineq:norm-var} may be rewritten as \(-\tilde{\eta}(\log \tilde{J}) + \tilde{\eta}(\log J) \leq 0 \). 
	To prove this inequality holds, we use the following standard lemma \cite[Lemma 3.3]{papo}:
	
	\begin{lemma} \label{lem:qqqp}
		Let \(q_{i}\), \(p_{i}\), \(1 \leq i \leq n\) be two families of strictly positive real numbers such that \(\sum
		q_{i} = \sum p_{i} = 1\). Then:
		\begin{equation*}
		- \sum_{i = 1}^{n} q_{i} \log \left( q_{i} \right) + \sum_{i = 1}^{n} q_{i} \log \left( p_{i} \right) \leq 0 \text.
		\end{equation*}
	\end{lemma}
	
	Let \(Q\) be the matrix-factor of \(\psi\) (recall Definition \ref{def:tr-type}). Employing Jensen's inequality to 
	the normalized trace, plus Lemma \ref{lem:qqqp}, it follows that, for all \(x \in \Omega_{A}\):
	\begin{equation*}
	- \sum_{\substack{i \in \Sigma \\ A(i,x_{1}) = 1}} \widehat{\operatorname{tr}}\, \left[ \psi_{ix} \log \left( \psi_{ix}(\mathbb{I}_{\mathcal{A}}) \right) \right] + \sum_{\substack{i \in \Sigma \\ A(i,x_{1}) = 1}} \widehat{\operatorname{tr}}\, \left[ \psi_{ix} \log \left( \varphi_{ix}(\mathbb{I}_{\mathcal{A}}) \right) \right] =
	\end{equation*}
	\begin{equation*}
	- \sum_{\substack{i \in \Sigma \\ A(i,x_{1}) = 1}} \widehat{\operatorname{tr}}\, \left[ Q(ix) \widehat{\operatorname{tr}}\, \log \left( Q(ix) \right) \right] + \sum_{\substack{i \in \Sigma \\ A(i,x_{1}) = 1}} \widehat{\operatorname{tr}}\, \left[ Q(ix) \widehat{\operatorname{tr}}\, \log \left( P(ix) \right) \right] \leq
	\end{equation*}
	\begin{equation*}
	- \sum_{\substack{i \in \Sigma \\ A(i,x_{1}) = 1}} \left[ \widehat{\operatorname{tr}}\, Q(ix) \right] \log \left( \widehat{\operatorname{tr}}\, Q(ix) \right) + \sum_{\substack{i \in \Sigma \\ A(i,x_{1}) = 1}} \left[ \widehat{\operatorname{tr}}\, Q(ix) \right] \log \left( \widehat{\operatorname{tr}}\, P(ix) \right) \leq 0
	\end{equation*}
	or, equivalently:
	\begin{equation*}
	- \widehat{\operatorname{tr}}\, (L_{\psi}\log \tilde{J})(x) + \widehat{\operatorname{tr}}\, (L_{\psi} \log J)(x) \leq 0 \text.
	\end{equation*}
	
	Because the eigenmeasure for \(\widehat{\operatorname{tr}}\, Q\) is positive:
	\begin{equation*}
	- \tilde{\eta}(L_{\psi}\log \tilde{J}) + \tilde{\eta}(L_{\psi} \log J) \leq 0 \text,
	\end{equation*}
	and because \(\tilde{\eta}\) is the eigenstate for \(L_{\psi}\):
	\begin{equation*}
	- \tilde{\eta}(\log \tilde{J}) + \tilde{\eta}(\log J) \leq 0 \text.
	\end{equation*}
\end{proof}

The above results suggest that the concept $h(\varphi,\eta,\sigma)$ of Definition  may capture the same structural role as entropy in the scalar case, and it is 
consistent with the interpretation of \(J\) as a generalized Jacobian. We also remark that because both inequalities
hold for any trace-type potentials inducing the given states, they would also hold for their infima over such
potentials, which would in turn be a function of the states themselves. However, Example \ref{ex:max-ent} and the ones in Subsection \ref{ex:pauli} 
show that it is possible for a non-trace-type potential (\(p \neq 1\)) to share eigenstates with a potential of trace-type (\(p = 1\)).

\section{Relation to previous literature} \label{rela}

The results of the present work touch upon some topics also addressed in \cite{ca-la} and \cite{ye}, both of which consider
transfer operators associated with matrix-valued potentials under given dynamical frameworks, but they are not included there. Our main theorem is
much more akin to the main theorems of \cite{ye}. We will elaborate on that.

In \cite{ca-la} results are framed in the context of smooth dynamical systems on compact differentiable manifolds. There,
the base map is assumed to be a covering map, which ensures that each point has finitely many local inverses—analogous
to inverse branches in symbolic dynamics. In contrast, \cite{ye} considers an iterated function system (IFS) acting on
a compact attractor via contractive mappings, where the symbolic space is replaced by a compact metric space.

In \cite{ye} it is assumed that the potential is positive (assumption (H1)) and moreover that, for each point, the family of 
matrix-valued functions given by the potential along its preimages satisfy a primitivity condition (assumption (H2)). 
On the other hand, \cite{ca-la} assumes invertibility of the potential matrices and does not require primitivity.

Our work is closer to \cite{ye} (but different), in the sense  that our functions also take values in a finite-dimensional space, and our potentials
are required to be positivity-improving, which is stronger than the combination (H1) plus (H2). However, the positivity 
aspect of our hypothesis refers to the intrinsic cone of positive elements in the C*-algebra \(\mathcal{A}\), which is
fundamentally different from the coordinate-wise positive cone. For example, in \(M_{2 \times 2}(\mathbb{R})\) a matrix 
is positive if and only if it is Hermitian and has non-negative eigenvalues, a condition that is nonlinear in the matrix 
entries and richer than positivity in \(\mathbb{R}^{4}\) (which would just mean positivity of each coordinate viewed as an element of \(\mathbb{R}^{4}\)).

As such, the positivity-improving hypothesis in our setting implies distinct structural constraints that are not
captured by simply embedding the algebra into a vector space. Consequently, our results cover a genuinely new class
of examples. This distinction, although subtle, also has implications for the structure of eigenstates and the
behavior under iteration of the associated transfer operators. The examples we present here evidence such a distinction.

Besides this distinction, it is also worthwhile to emphasize that, in the case of a noncommutative C*-algebra
\(\mathcal{A}\), our eigenstate is not necessarily a (vector-valued) eigenmeasure, as in \cite{ye}. Rather, it is an
``\(\mathcal{A}\)-state-valued'' measure. Not only is the positivity condition different, but the normalization condition is
different as well (and accordingly).

We believe that a general formulation in the setting of order-unit spaces could clarify these distinctions, potentially
allowing one to isolate the minimal structure needed for our results. Within such a framework, C*-algebras would appear 
as natural examples where one still manages to handle richer cone geometries explicitly. We leave such generalizations for 
further work.

\section{Some more Examples} \label{some}

\subsection{The constant potential and the maximal entropy measure} \label{some1}

Assume \(\Gamma : \mathcal{A} \to \mathcal{A}\) is positivity-improving, $\mathcal{A}=  M_{n \times n}(\mathbb{R}) $, and satisfies \(\Gamma (\mathbb{I}) = \mathbb{I}\).
Moreover, assume that trace $\Gamma (a)= $ trace $a$, for any  matrix $a\in   M_{n \times n}(\mathbb{R}) $.  We denote $\xi(a)=$ trace $a$, in this way for any $a \in M_{n \times n}(\mathbb{R})$ we get
$$ \xi(a)=\xi ( \Gamma(a) ).$$

Set \(\Omega_{A} = \left\{1, 2, \cdots, k\right\}^{\mathbb{N}}\)
the full shift on \(k\) symbols. 

 In most of our examples, $\xi=$ trace will help to define the eigenstate; but, for instance, when $\mathcal{A}= \mathbb{C}^d$, this is not  exactly the case.

Now, we define the potential \(\Phi : \Omega_{A} \to \mathfrak{L}(\mathcal{A})\) 
in such way that for any $x$ we get that $\Phi(x) =\Gamma.$  We call such $\Phi$ of the constant potential. In this case
 \(\Phi(\mathbb{I}_{\mathcal{A}}) = \mathbb{I}_{\mathcal{A}}\). We denote by  $L_{ \frac{1}{k}\Phi}$ the associated Ruelle operator.
 
 Under our assumptions, for any $g:\Omega \to M_{n \times n}(\mathbb{R})$ we get that
 $$ \xi(g(x)) = \xi (\Phi (g(x)  )= \xi (\Gamma (g(x)).$$

We claim that the eigenstate \(\eta\) given by Theorem \ref{thm:ruelle-normal} acting on continuous
functions \(g \in \mathcal{C}(\Omega_{A}; \mathcal{A})\)  is given by the expression:
\begin{equation} \label{eq:const-pot}
\eta(g) = \int \xi(g(x)) \operatorname{d} \! \mu(x) \text,
\end{equation}
where \(\mu\) denotes the maximal entropy measure.

This is so because:
\begin{equation*} 
\eta(L_{\Phi} g) = \int \xi\left(\left( L_{\Phi} g\right)(y)\right) \operatorname{d} \! \mu(y)
= \int \xi\left(\sum_{i = 1}^{k} \Phi\left(g(iy)\right)\right) \operatorname{d} \! \mu(y)
\end{equation*}
\begin{equation*} 
= \int \sum_{i = 1}^{k} \xi\left(\Phi\left(g(iy)\right)\right) \operatorname{d} \! \mu(y)
= \int \frac{1}{k} \sum_{i = 1}^{k} \xi\left(g(iy)\right) \operatorname{d} \! \mu(y) 
\end{equation*}
\begin{equation} \label{ute}
= \int \left( L_{\frac{1}{k}} \left(\xi \circ g\right)\right)(y) \operatorname{d} \! \mu(y) 
= \int \xi\left(g(x)\right) \operatorname{d} \! \mu(x) 
= \eta(g) 
\end{equation}
where \(L_{\frac{1}{k}}\) denotes the \emph{classical} normalized transfer operator associated with the maximal
entropy measure.

\subsection{An example related to Pauli matrices} \label{ex:pauli}

\begin{example}
Let \( \Omega_{A} = \{1,2\}^\mathbb{N} \) be the full shift on two symbols and  $\mathcal{A}=  M_{2 \times 2}(\mathbb{R})$, and \(0 < p \leq 1\).  We want to define, for each value $p$ a potential $\varphi^p.$
Consider the constant potential \(\varphi^{p}: M_{2 \times 2}(\mathbb{R}) \to M_{2 \times 2}(\mathbb{R}) \) valued in the depolarizing
channel, that is:
\begin{equation*}
\varphi^{p}(g(x)) := \frac{1}{2} \left((1 - p) g(x) + p \left(\widehat{\operatorname{tr}}\, g(x)\right) \operatorname{Id} \right) 
\end{equation*}
\begin{equation}\label{depo}
= \frac{1}{2} \left((1 - p) g(x) + \frac{p}{3} \left(X g(x)
X + Y g(x) Y + Z g(x) Z\right)\right) \text,
\end{equation}
where \(X\), \(Y\), and \(Z\) are the Pauli matrices (see \cite{nich}):
$$
X=\left(
\begin{array}{cc}
0 & 1\\
1 & 0
\end{array}
\right),\,\,
Y=\left(
\begin{array}{cc}
0& -i\\
i & 0
\end{array}
\right),\,\,
\text{and}\,\,
Z=\left(
\begin{array}{cc}
1 & 0\\
0 & -1
\end{array}
\right).$$

The traces of all these matrices is equal to zero,  then the complex nature of the Pauli matrices will not interfere in our reasoning. The  depolarizing
channel present such feature (see \cite{nich}).

\eqref{depo} is the normalized qubit depolarizing channel, and it defines a completely positive, unital, and positivity-improving
map for every \(0 < p \leq 1\) \cite[Subsection 8.3.4]{nich}. The action of its associated transfer operator is:
\begin{equation} \label{gg5}
\left(L_{\varphi^{p}} g\right)(y) = \sum_{j=1}^{2} \varphi(g(j y)) 
\end{equation}
\begin{equation*}
= \left(1 - p\right) \left(\frac{1}{2} \sum_{j=1}^{2} g(j y)\right) + p \left(\frac{1}{2} \sum_{j=1}^{2} \left( \widehat{\operatorname{tr}}\, g( jy) \right) \operatorname{Id} \right) \text.
\end{equation*}

By iteration, one finds for each value $n$ the general expression:
\begin{equation*}
(L_{\varphi^{p}}^{n} g)(y) = (1 - p)^{n} \overline{g}^{(n)}(y)
+ \left(1 - (1 - p)^{n}\right) \overline{\widehat{\operatorname{tr}}\, g}^{(n)}(y) \operatorname{Id}
\text,
\end{equation*}
where:
\begin{align*}
\overline{g}^{(n)}(y) &:= \frac{1}{2^n} \sum_{i_1, \dots, i_n} g(i_1 \dots i_n y) \text{, and:} \\
\overline{\widehat{\operatorname{tr}}\, g}^{(n)}(y) &:= \frac{1}{2^n} \sum_{i_1, \dots, i_n} \widehat{\operatorname{tr}}\, g(i_1 \dots i_n y) \text.
\end{align*}

Notice how \(\overline{\widehat{\operatorname{tr}}\, g}^{(n)}\) coincides with the \(n^{\operatorname{th}}\)
iterate of the \emph{scalar} valued function \(\widehat{\operatorname{tr}}\, g\) with respect to the classical
transfer operator associated with the maximal entropy measure. Therefore, as \( n \to +\infty \), we have
\( (1 - p)^n \to 0 \) and \( \overline{\widehat{\operatorname{tr}}\, g}^{(n)}(y) \to \int \widehat{\operatorname{tr}}\, g(x) \operatorname{d} \! \mu(x) \).
Hence:
\begin{equation*}
\lim_{n \to +\infty} (L_{\varphi^{p}}^n g)(y) = \left( \int \widehat{\operatorname{tr}}\, g(x) \operatorname{d} \! \mu(x) \right) \operatorname{Id} \text, \quad \forall y \in \left\{1,2\right\}^{\mathbb{N}} \text.
\end{equation*}
Item 3. (b) in Theorem \ref{thm:ruelle-normal} then ensures this convergence is uniform and that the eigenstate \(\eta\) is given by:
\begin{equation} \label{gg6}
\eta(g) := \int \widehat{\operatorname{tr}}\, g(x) \operatorname{d} \! \mu(x).
\end{equation}
\end{example}

\subsection{Potentials taking values in positivity-improving \\ linear mappings of \(d \times d\) complex matrices} \label{ok2}

In this section, we consider the case when \(\mathcal{A} = M_{d \times d}(\mathbb{R})\); some states that are in principle non-commutative in nature will be related to classical equilibrium measures as in \cite{papo}.

Now we
are going to focus only on potentials taking values in positivity-improving linear mappings of the form:
\begin{equation*}
g(x) \mapsto \varphi_{x}(g(x)) := \left( \widehat{\operatorname{tr}}\, g(x) \right) \begin{bmatrix} f_{1}(x) & 0 & \cdots & 0 \\ 0 & f_{2}(x) & & \vdots \\ \vdots & & \ddots & \\ 0 & \cdots & & f_{d}(x) \end{bmatrix} \text{,}
\end{equation*}
where \(\widehat{\operatorname{tr}} := \frac{1}{d} \operatorname{tr} : M_{d \times d}(\mathbb{R}) \to \mathbb{R}\) 
denotes the normalized trace functional, and \(\mathcal{C}(\Omega_{A}; \mathbb{R}) \ni f_{1}, f_{2}, \cdots, f_{d} > 0\).

For one such potential, the action of the transfer operator on a given $g \in \mathcal{C}(\Omega_{A}; M_{d \times d}(\mathbb{R}))$ becomes: 
\begin{equation*}
\left( L_{\varphi} g \right)(y) = \sum_{\substack{i \in \Sigma \\ A(i,y_{1}) = 1}} \left( \widehat{\operatorname{tr}}\, g(iy) \right) \begin{bmatrix} f_{1}(iy) & 0 & \cdots & 0 \\ 0 & f_{2}(iy) & & \vdots \\ \vdots & & \ddots & \\ 0 & \cdots & & f_{d}(iy) \end{bmatrix} \text.
\end{equation*}

Next, we will present a series of specific examples that will demonstrate the novelties that follow from our theoretical results.

\medskip

\begin{example} \label{ex:max-ent}
	Let \(k =  2\), $\mathcal{A}=M_{2 \times 2}(\mathbb{R}) $, and  the shift of finite type is defined via \(A = \begin{bmatrix} 1 & 1 \\ 1 & 1 \end{bmatrix}\); that is the full shift. In this case,  \(\Sigma = \left\{1, 2\right\}\)
	and \(\Omega_{A} = \left\{1, 2\right\}^{\mathbb{N}}\) is the full shift on two symbols.
	Also, let \(f_{1} = f_{2} = \frac{1}{2}\). 
	Then, the positivity-improving linear self-mapping of \(M_{2 \times 2}(\mathbb{R})\) in question is (half) a very 
	well-known quantum channel: the (degenerate) depolarizing channel (where, in the language of \cite[Subsection 8.3.4]{nich}, \(p = 1\)). 
	More explicitly, \(\varphi_{x}\) does not depend on \(x\in \Omega_A\) and is such that:
	\begin{equation}\label{43}
	g(x) \mapsto \left( \widehat{\operatorname{tr}}\, g(x) \right) \begin{bmatrix} \frac{1}{2} & 0 \\ 0 & \frac{1}{2} \end{bmatrix} \text{.}
	\end{equation}
	We denote the corresponding potential by \(\frac{\operatorname{Id}}{2} \widehat{\operatorname{tr}}\). The action of its associated transfer operator is: for $y\in \Omega$
	\begin{equation*}
	\left( L_{\frac{\operatorname{Id}}{2} \widehat{\operatorname{tr}}} g \right)(y) = \frac{1}{2} \left( \widehat{\operatorname{tr}} \Big( g(1y) + g(2y) \Big) \right) \operatorname{Id} \text.
	\end{equation*}
	
	Notice that after the first iteration, we (essentially) recover the classical transfer operator with constant
	normalized potential. In some sense, it is as if such a transfer operator only ``sees'' the trace, which is a \emph{scalar}
	function of \(x\). However, there is some differences; in the classical case (\textit{i.e.}, when \(\mathcal{A} = \mathbb{R}\), or \(d = 1\)), the eigenstate 
	of the dual of the transfer operator is a probability measure, and acts on \emph{scalar} functions \(f \in \mathcal{C}(\Omega_{A}; \mathbb{R})\) 
	simply by integration, while, in the present example, our observables are matrix-valued functions \(g \in \mathcal{C}(\Omega_{A}; M_{2 \times 2}(\mathbb{R}))\), 
	and the eigenstate \(\eta\) is no longer a measure, but rather a linear functional 
	\(\eta : \mathcal{C}(\left\{1,2\right\}^{\mathbb{N}}; M_{2 \times 2}(\mathbb{R})) \to \mathbb{R}\) defined by:
	\begin{equation*}
	g \mapsto \eta(g) := \int \widehat{\operatorname{tr}}\, g(x) \operatorname{d} \! \mu(x) \text,
	\end{equation*}
	where \(\mu\) denotes the maximal entropy measure associated with the classical full shift.
	
	In other words, while the measure \(\mu\) still plays a role, it is no longer the eigenstate itself. Instead, the
	eigenstate is the composition of integration against \(\mu\) with evaluation via the (normalized) trace. This
	reflects the operator-valued nature of the observables and the fact that scalar probability measures are
	insufficient to describe equilibrium in the noncommutative setting.
\end{example}

\begin{example} \label{ex:first-coord}
	Again, let \(\Omega_{A}\) be the full shift on two symbols and also \(d = 2\), $\mathcal{A}=M_{2 \times 2}(\mathbb{R}) $. Consider the functions \(f_{1}\) and \(f_{2}\)
	depending only on the first coordinate, such that:
	\begin{equation*}
	\begin{aligned}
	f_{1}(1x) := \lambda_{1} \text{, } f_{1}(2x) := 1 - \lambda_{1} \text, \\
	f_{2}(1x) := \lambda_{2} \text{, } f_{1}(2x) := 1 - \lambda_{2} \text, \\
	\end{aligned} \quad \forall x \in \Omega_{A} \text{, and } 0 < \lambda_{1} \text, \lambda_{2} < 1 \text.
	\end{equation*}
	
	Then, the eigenstate \(\eta : \mathcal{C}(\left\{1,2\right\}^{\mathbb{N}}; M_{2 \times 2}(\mathbb{R})) \to \mathbb{R}\)
	is:
	\begin{equation*}
	T \mapsto \eta(T) := \int \widehat{\operatorname{tr}}\, g(x) \operatorname{d} \! \mu_{\operatorname{avg}}(x) \text,
	\end{equation*}
	where \(\mu_{\operatorname{avg}}\) is the eigenmeasure for the classical transfer operator associated with the
	\emph{scalar} potential \(\frac{f_{1} + f_{2}}{2}\). Notice that this case generalizes example \ref{ex:max-ent}.
\end{example}

\begin{example} \label{ex:diagonals}
	Let \(\Omega_{A} \subseteq \left\{1, 2, \cdots, k\right\}^{\mathbb{N}}\) be any aperiodic subshift of finite type, $\mathcal{A}=  M_{d \times d}(\mathbb{R})$,
	and \(\mathcal{C}(\Omega_{A}; \mathbb{R}) \ni f_{1}, f_{2}, \cdots, f_{d} > 0\) a \(d\)-tuple of normalized scalar potentials. 
	Notice that their average is also a normalized scalar potential.
	Then, the corresponding potential \(\varphi_{f_{1}, f_{2}, \cdots, f_{d}} : \Omega_{A} \to \mathcal{L}(M_{d \times d}(\mathbb{R}))\)
	is also normalized. We claim that  the eigenstate \(\eta : \mathcal{C}(\Omega_{A}; M_{d \times d}(\mathbb{R})) \to \mathbb{R}\)
	mentioned in Theorem \ref{thm:ruelle-normal} is:
	\begin{equation*}
	g \mapsto \eta(g) := \int \widehat{\operatorname{tr}}\, g(x) \operatorname{d} \! \mu_{\operatorname{avg}}(x) \text,
	\end{equation*}
	where \(\mu_{\operatorname{avg}}\) is the eigenmeasure for the classical transfer operator associated with the
	\emph{scalar} potential \(\frac{1}{d} \sum_{i =1}^{d} f_{i}\). Notice how this generalizes example \ref{ex:first-coord}.
\end{example}

\begin{proof}
	\begin{equation*}
	\eta(L_{\varphi_{f_{1}, f_{2}, \cdots, f_{d}}} g) = \int \widehat{\operatorname{tr}}\, \left[\left(L_{\varphi_{f_{1}, f_{2}, \cdots, f_{d}}} g\right)(y)\right] \operatorname{d} \! \mu_{\operatorname{avg}}(y) 
  \end{equation*}
	\begin{equation*}
	= \int \widehat{\operatorname{tr}}\, \left(\sum_{\substack{i \in \Sigma \\ A(i,y_{1}) = 1}} \left( \widehat{\operatorname{tr}}\, g(iy) \right) \begin{bmatrix} f_{1}(iy) & 0 & \cdots & 0 \\ 0 & f_{2}(iy) & & \vdots \\ \vdots & & \ddots & \\ 0 & \cdots & & f_{d}(iy) \end{bmatrix}\right) \operatorname{d} \! \mu_{\operatorname{avg}}(y) 
  \end{equation*}
  \begin{equation*}
	= \int \sum_{\substack{i \in \Sigma \\ A(i,y_{1}) = 1}} \left( \widehat{\operatorname{tr}}\, g(iy) \right) \left(\frac{1}{d} \sum_{j = 1}^{d} f_{j}(iy)\right) \operatorname{d} \! \mu_{\operatorname{avg}}(y) 
	= \int \Big(L_{\operatorname{avg}} \widehat{\operatorname{tr}}\, g\Big)(y) \operatorname{d} \! \mu_{\operatorname{avg}}(y) 
  \end{equation*}
  \begin{equation*}
	= \int \widehat{\operatorname{tr}}\, g(x) \operatorname{d} \! \mu_{\operatorname{avg}}(x) 
	= \eta(g) \text.
	\end{equation*}
\end{proof}

\begin{example} \label{spi1}
	Let $\mathcal{A}= M_{2 \times 2}(\mathbb{R})$ and $\Omega=\{-1,1\}^\mathbb{N}$ .
	We can consider that $-1$ is associated with the spin $-$ and $1$ with the spin $+$.

Consider the class of functions  $\varphi : \Omega=\{-1,1\}^\mathbb{N} \to \mathcal{L}(M_{2 \times 2}(\mathbb{R}))$ 
  such that there exist strictly positive matrices $A_{-1}, A_{1} \in M_{2 \times 2}(\mathbb{R})$, satisfying for any $y\in \Omega$,
  \begin{equation} \label{seaf} \varphi_{(1,y)} = \left( \widehat{\operatorname{tr}}\,( \cdot ) \right) A_{-1} \,\, \text{and}\,\, \varphi_{(2,y)} = \left( \widehat{\operatorname{tr}}\,( \cdot ) \right) A_{1} \text,
\end{equation}
$$A_{-1} + A_{1} =I \text,$$

  Defining:
  \begin{equation*}
    J_{\varphi}(-1,y) := \widehat{\operatorname{tr}}\, A_{-1} \text{ , and }
    J_{\varphi}(1,y) := \widehat{\operatorname{tr}}\, A_{1} \text,
  \end{equation*}
	it follows that for any $y\in \Omega$, $J_\varphi(-1,y)+ J_\varphi(1,y)=1$. Denote by $\mu_{J_\varphi}$ the (classical) 
  equilibrium probability associated with the normalized potential $J_\varphi$.
	
	We denote by $g$ a general continuous function $g: \Omega \to  M_{2 \times 2}(\mathbb{R}) $.
	
  The action of the transfer operator is:
	\begin{equation} \label{gg1} g \,\to L_\varphi(g)(y) =\frac{1}{2} (\,\widehat{\operatorname{tr}}\,  g(1,y)\,) \, A_1+ \frac{1}{2} (\,\widehat{\operatorname{tr}}\,  g(-1,y)\,) \, A_{-1}.
	\end{equation}
	
	Note that the potential $ \varphi $ is positivity-improving.
	
	We claim that $\eta_\varphi$ such that
	\begin{equation} \label{gg2}  \eta_\varphi(T)=\int\widehat{\operatorname{tr}} \, (g(x))\,\, d \mu_{J_\varphi} (x)
	\end{equation}
	is the eigenstate for such a Ruelle operator. Indeed, given $g$
	$$\eta_\varphi(L_\varphi(T)  )= \int\widehat{\operatorname{tr}} \,(L_\varphi (g(x))\,\, d \mu_{J_\varphi} (x)= $$
	$$\int \widehat{\operatorname{tr}} \,[\,(\,\widehat{\operatorname{tr}}\, \frac{1}{2} g(1,x)\,) \, A_1+  (\,\widehat{\operatorname{tr}}\, \frac{1}{2} g(-1,x)\,) \, A_{-1}\,]\, d \mu_{J_\varphi}(x)=  $$
	$$\int  \,\,\widehat{\operatorname{tr}}\,  g(1,x)\, \,\widehat{\operatorname{tr}}\,(\frac{1}{2} A_1)+  \,\widehat{\operatorname{tr}}  g(-1,x)\, \, \widehat{\operatorname{tr}}\,(\frac{1}{2} A_{-1})\,\, d \mu_{J_\varphi}(x)=  $$
	$$\int  \,(\,\widehat{\operatorname{tr}}\,  g(1,x)\,) \,J(1,x))+  (\,\widehat{\operatorname{tr}}\,  g(-1,x)\,) \, J(-1,x))\,\, d \mu_{J_\varphi}(x)=  $$
	\begin{equation} \label{itre}\int \widehat{\operatorname{tr}} \, (g(x))\, d \mu_{J_\varphi} (x)=\eta_\varphi (g).
	\end{equation}
	
	\medskip
	
	We may call each  $\eta_\varphi$ obtained from a $\varphi$ as in \eqref{seaf}  of Gibbs  eigenstate (see Section \ref{Pre}).
	
	In consonance with Section \ref{Pre},  for each $\varphi$  the corresponding entropy of $\eta_\varphi$
	is:
  $$h(\eta, \varphi)= \eta_\varphi (\log \varphi)= \int \,\,\widehat{\operatorname{tr}} \, (\log \varphi_{x}(\operatorname{Id})) \, d \mu_{J_\varphi} (x).$$

\end{example}

\medskip

\begin{example} We will define a potential $\varphi$  acting on the family of continuous functions $g:\Omega \to M_{2 \times 2}(\mathbb{R})=\mathcal{A}$,  where $\Omega=\Omega_A=\{1,2\}^\mathbb{N}$ is the full shift space. Consider a Lipschitz continuous map $x \in \Omega_{A} \to P(x)$, where $P(x)$ is double stochastic with positive entries, and the
	corresponding matrices

	\begin{equation*}
	P_1(x) = \begin{bmatrix}
	\sqrt{p_{11}(x)} & 0 \\
	0 & \sqrt{p_{21}(x)}
	\end{bmatrix} \text, \quad
	P_2 (x)= \begin{bmatrix}
	\sqrt{p_{12}(x)} & 0 \\
	0 & \sqrt{p_{22}(x)}
	\end{bmatrix} \text,
	\end{equation*}
	and
	\begin{equation*}
	K (x)= 
	\begin{bmatrix}
	\sqrt{\pi_1(x)} & 0 \\
	0 & 0
	\end{bmatrix} \otimes P_1(x)
	+
	\begin{bmatrix}
	0 & 0 \\
	0 &\sqrt{\pi_2(x)}
	\end{bmatrix} \otimes P_2(x), 
	\end{equation*}
	where $(\pi_1(x,),\pi_2(x))=(\frac{1}{2},\frac{1}{2})$ is the stationary vector for the double stochastic matrix $P(x)$.
	
	\medskip

	Consider the transformation $W $ such that
	$$W(\begin{bmatrix}
	a_{11} & a_{12} \\
	a_{21} & a_{22}
	\end{bmatrix})  =\begin{bmatrix}
	a_{11}+ a_{22} & 0 \\
	0 & a_{11} +a_{22}
	\end{bmatrix} ,$$

	For each $x$, consider potentials of the form 
	$$x \mapsto \varepsilon_x(W(a(x)) \otimes \operatorname{Id})=  \operatorname{Tr}_1 \left( K(x) (W(a(x) )\otimes \operatorname{Id}) K (x)\right) .$$
	

	We want to determine the eigenstate for  $x \to \varepsilon_x(W(a(x)) \otimes \operatorname{Id})$
	
	\medskip
	
	Denote by 
	$$a(x)=\begin{bmatrix}
	a_{11}(x) & a_{12}(x) \\
	a_{21}(x) & a_{22}(x)
	\end{bmatrix}$$
	a general continuous map on $x\in \Omega$. For simplification, we will sometimes omit the dependence on $x$.

	For the fixed family $P(x)$
	consider the variable function  $a$. In this case
	\begin{equation*}
	\varepsilon_x(W(a(x)) \otimes \operatorname{Id}) := \operatorname{Tr}_1 \left( K^{*}\, (W(a(x)) \otimes \operatorname{Id}) \,K \right) 
  \end{equation*}
  \begin{equation*}
	= \operatorname{Tr}_1 \left(
    \left( \begin{array}{l}
	\begin{bmatrix}
	\sqrt{\pi_1} & 0 \\
	0 & 0
	\end{bmatrix} \otimes P_1 + \\ +
	\begin{bmatrix}
	0 & 0 \\
	0 & \sqrt{\pi_2}
	\end{bmatrix} \otimes P_2
  \end{array} \right)
	(W(a) \otimes \operatorname{Id})
  \left( \begin{array}{l}
	\begin{bmatrix}
	\sqrt{p_1} & 0 \\
	0 & 0
    \end{bmatrix} \otimes P_1 + \\ +
	\begin{bmatrix}
	0 & 0 \\
	0 & \sqrt{\pi_2}
	\end{bmatrix} \otimes P_2
  \end{array} \right)
	\right) 
  \end{equation*}
  \begin{equation*}
	= \operatorname{Tr}_1 \left(
	\begin{bmatrix}
	\sqrt{\pi_1} & 0 \\
	0 & 0
	\end{bmatrix} W(a)
	\begin{bmatrix}
	\sqrt{\pi_1} & 0 \\
	0 & 0
	\end{bmatrix} \otimes P_1^2
	+
	\begin{bmatrix}
	0 & 0 \\
	0 & \sqrt{\pi_2}
	\end{bmatrix}W( a)
	\begin{bmatrix}
	0 & 0 \\
	0 & \sqrt{\pi_2}
	\end{bmatrix} \otimes P_2^2
	\right) 
  \end{equation*}
  \begin{equation*}
	= \operatorname{Tr}_1 \left(
	\begin{bmatrix}
	(a_{11}+ a_{22} )\,\pi_1& 0 \\
	0 & 0
	\end{bmatrix} \otimes P_1^2
	+
	\begin{bmatrix}
	0 & 0 \\
	0 & (a_{11} + a_{22})\, \pi_2
	\end{bmatrix} \otimes P_2^2
	\right) 
  \end{equation*}
  \begin{equation*}
	= (a_{11} + a_{22}) \,\pi_1\, P_1^2 + (a_{11} +a_{22}) \pi_2 \, P_2^2=(a_{11} (x)+ a_{22}(x))\, \begin{bmatrix}
	\frac{1}{2} &0\\
	0 & \frac{1}{2}
	\end{bmatrix} \text.
	\end{equation*}
	
	In the similar way as in Example \ref{cf} (see also \eqref{43} in Example \ref{ex:max-ent}) we set for any $g$
	$$L_\varphi (g)(x)=$$
	  \begin{equation} \label{gg10}(g_{11}(1 x) \begin{bmatrix}
	\frac{1}{2} &0\\
	0 & \frac{1}{2}
	\end{bmatrix} + g_{22}(1x) \begin{bmatrix}
	\frac{1}{2} &0\\
	0 & \frac{1}{2}
	\end{bmatrix} )+   (g_{11}(2 x) \begin{bmatrix}
	\frac{1}{2} &0\\
	0 & \frac{1}{2}
	\end{bmatrix} + g_{22}(2x) \begin{bmatrix}
	\frac{1}{2} &0\\
	0 & \frac{1}{2}
	\end{bmatrix})
	\end{equation}
	  
	 The eigenstate is 
	 $$ g \to \frac{1}{2} \int (a_{11} (x)+ a_{22}(x))\, d \mu (x)= \frac{1}{2} \int \left( \widehat{\operatorname{tr}} \Big( g(x) \Big) \right) d \mu(x).$$
    Notice that this eigenstate is the same as the one coming from the potential of Example \ref{ex:max-ent}.
	
\end{example}

\subsection{Potentials taking values in  positive matrices}
\label{ok3}

In this section we consider the case when \(\mathcal{A} = \mathbb{R}^{N}\). This means we are concerned with \emph{vector}-valued 
functions \(g \in \mathcal{C}(\Omega_{A}; \mathbb{R}^{N})\). The multiplicative element in \(\mathcal{A} = \mathbb{R}^{N}\) is the vector $(1,1,...,1)$. The positivity-improving linear self-mappings of
\(\mathbb{R}^{N}\) are precisely the ones corresponding to \(N \times N\) square-matrices with strictly positive
entries. Therefore, the potentials we will be interested in are functions of the form \(\varphi : \Omega_{A} \to M_{N \times N}(\mathbb{R})\)
such that:
\begin{equation*}
\varphi_{x} := \begin{bmatrix} \varphi_{x_{11}} & \varphi_{x_{12}} & \cdots & \varphi_{x_{1N}} \\
\varphi_{x_{21}} & \varphi_{x_{22}} & \cdots & \varphi_{x_{2N}} \\
\vdots & \vdots & \ddots & \vdots \\
\varphi_{x_{N1}} & \varphi_{x_{N2}} & \cdots & \varphi_{x_{NN}} \\
\end{bmatrix} \text{ , where: } 
\begin{aligned} 
\varphi_{x_{\ell m}} > 0 \\
\forall \, 1 \leq \ell, m \leq N \text.
\end{aligned}
\end{equation*}

For one such potential, the action of the transfer operator becomes:
\begin{equation*}
\left( L_{\varphi} g \right)(y) = \sum_{\substack{i \in \Sigma \\ A(i,y_{1}) = 1}} \begin{bmatrix} \varphi_{iy_{11}} & \varphi_{iy_{12}} & \cdots & \varphi_{iy_{1N}} \\
\varphi_{iy_{21}} & \varphi_{iy_{22}} & \cdots & \varphi_{iy_{2N}} \\
\vdots & \vdots & \ddots & \vdots \\
\varphi_{iy_{N1}} & \varphi_{iy_{N2}} & \cdots & \varphi_{iy_{NN}} \\
\end{bmatrix} \begin{bmatrix} g_{1}(iy) \\ g_{2}(iy) \\ \vdots \\ g_{N}(iy) \end{bmatrix} 
\end{equation*}
\begin{equation*}
=\sum_{\substack{i \in \Sigma \\ A(i,y_{1}) = 1}} \sum_{\ell = 1}^{N}
\begin{bmatrix} \varphi_{iy_{1\ell}} g_{\ell}(iy) \\ \varphi_{iy_{2\ell}} g_{\ell}(iy) \\ \vdots \\ \varphi_{iy_{N\ell}} g_{\ell}(iy) \end{bmatrix} \text.
\end{equation*}

\begin{example}
	Let \(k = 2\), and \(A = \begin{bmatrix} 1 & 1 \\ 1 & 1 \end{bmatrix}\). That is, \(\Sigma
	= \left\{1, 2\right\}\)
	and \(\Omega_{A} = \left\{1, 2\right\}^{\mathbb{N}}\) is the full shift on two symbols. Also, let \(\varphi_{x_{\ell m}} = \frac{1}{2N}\),
	that is, \(\varphi_{x} = \frac{1}{2N} J\), where \(J\) is the $N$ by $N$  matrix with all entries equal to $1$; therefore does not depend on \(x\).
	Then, the action of the associated normalized transfer operator becomes: given $g=(g_1,g_2,...,g_N)$
	\begin{equation*}
	\left( L_{\frac{1}{2N} J} g\right)(y) = \frac{1}{2N} \sum_{i, \ell = 1}^{2, N} \begin{bmatrix} g_{\ell}(iy) \\ g_{\ell}(iy) \\ \vdots \\ g_{\ell}(iy) \end{bmatrix} \text,
	\end{equation*}
	and the corresponding eigenstate \(\eta : \mathcal{C}(\Omega_{A}; \mathbb{R}^{N}) \to \mathbb{R}\) is given by:
	\begin{equation*}
	\eta(g) = \frac{1}{N} \sum_{\ell = 1}^{N} \int g_{\ell}(x) \operatorname{d} \! \mu(x) \text,
	\end{equation*}
	where \(\mu\) denotes the maximal entropy measure.
\end{example}

\begin{example}
	Again, let \(\Omega_{A} = \left\{1, 2\right\}^{\mathbb{N}}\) be the full shift on two symbols. This time, let
	\(\varphi_{x} = \frac{1}{2} \left(\frac{p}{N} J + (1 - p) \operatorname{Id} \right)\), not depending on \(x\).
	The \(n^{\operatorname{th}}\) power of \(\left(\frac{p}{N} J + (1 - p) \operatorname{Id} \right)\) is equal to 
	\(\left(\frac{1 - (1 - p)^{n}}{N} J + (1 - p)^{n} \operatorname{Id} \right)\). Therefore, the
	\(n^{\operatorname{th}}\) iterate of the corresponding transfer operator is:
	\begin{equation*}
	\left( L_{\varphi}^{n} g \right)(y) = \frac{1}{2^{n}} \sum_{i_{1}, i_{2}, \cdots, i_{n} = 1}^{2} \left(\frac{1 - (1 - p)^{n}}{N} J + (1 - p)^{n} \operatorname{Id}\right) g(i_{1}i_{2}\cdots i_{n}y) \text.
	\end{equation*}
	
	Taking the limit \(n \to +\infty\), we see the action of the above operator becomes the action  of \(L_{\frac{1}{2N} J}\).
	Therefore, they determine the same eigenstate \(\eta\).
\end{example}

\phantom. \hrule \phantom.

William M. M. Braucks  (braucks.w@gmail.com)
\smallskip

Artur O. Lopes (arturoscar.lopes@gmail.com)
\smallskip

Inst. Mat. Est. - UFRGS - Porto Alegre, Brazil

\end{document}